\documentclass[final,leqno]{siamltex}

\usepackage[utf8]{inputenc}
\usepackage[english]{babel}
\usepackage{amsmath, amssymb}

\usepackage[hang]{caption}
\usepackage{graphicx}
\usepackage{subfig}
\usepackage{color}

\usepackage{nicefrac}

\usepackage{multirow}
\usepackage{dcolumn}
\newcolumntype{d}[1]{D{.}{.}{#1}}

\usepackage{hyperref}
\usepackage{bookmark}
\hypersetup{
  pdftitle    = {Solving the Monge-Amp\`ere Equations for the Inverse Reflector Problem},
  pdfauthor   = {Kolja Brix, Yasemin Hafizogullari, and Andreas Platen},
  pdfkeywords = {Inverse reflector problem, elliptic Monge-Amp\`ere equation, B-spline collocation method, Picard-type iteration},
  pdfborder={0 0 0.5}
}

\newcommand{\R}{\ensuremath{\mathbb{R}}}
\newcommand{\Z}{\ensuremath{\mathbb{Z}}}
\newcommand{\N}{\ensuremath{\mathbb{N}}}
\newcommand{\DD}{\ensuremath{\textnormal{D}}}
\newcommand{\dd}{\ensuremath{\textnormal{d}}}

\newcommand{\norm}[1]{\ensuremath{\|#1\|}}
\newcommand{\normlr}[1]{\ensuremath{\left\|#1\right\|}}
\renewcommand{\vec}[1]{\ensuremath{\mathbf{#1}}}

\newtheorem{remark}{Remark}
\newtheorem{problem}{Problem}
\newenvironment{ownproblem}[1]
  {\renewcommand\theproblem{#1}\problem}
  {\endproblem}

\author{Kolja Brix\footnotemark[4] \and
Yasemin Hafizogullari\footnotemark[4] \and
Andreas Platen\footnotemark[4]}

\date{\today}
\title{Solving the Monge-Amp\`ere Equations\\ for the Inverse Reflector Problem}

\begin{document}

\maketitle

\renewcommand{\thefootnote}{\fnsymbol{footnote}}

\footnotetext[4]{Institut f{\"u}r Geometrie und Praktische Mathematik, RWTH Aachen, Templergraben 55, 52056 Aachen, Germany
(\texttt{brix@igpm.rwth-aachen.de},
\texttt{yasemin.hafizogullari@rwth-aachen.de},
\texttt{andreas.platen@rwth-aachen.de}),
\url{http://www.igpm.rwth-aachen.de}.}

\renewcommand{\thefootnote}{\arabic{footnote}}

\begin{abstract}
The inverse reflector problem arises in geometrical nonimaging optics:
Given a light source and a target, the question is how to design a reflecting free-form surface such that a desired light density distribution is generated on the target, e.g., a projected image on a screen.
This optical problem can mathematically be understood as a problem of optimal transport and equivalently be expressed by a secondary boundary value problem of the Monge-Amp\`ere equation, which consists of a highly nonlinear partial differential equation of second order and constraints.
In our approach the Monge-Amp\`ere equation is numerically solved using a collocation method based on tensor-product B-splines, in which nested iteration techniques are applied to ensure the convergence of the nonlinear solver and to speed up the calculation.
In the numerical method special care has to be taken for the constraint: It enters the discrete problem formulation via a Picard-type iteration.
Numerical results are presented as well for benchmark problems for the standard Monge-Amp\`ere equation as for the inverse reflector problem for various images. The designed reflector surfaces are validated by a forward simulation using ray tracing.
\end{abstract}

\begin{keywords}
Inverse reflector problem, elliptic Monge-Amp\`ere equation, B-spline collocation method, Picard-type iteration
\end{keywords}

\begin{AMS}
35J66, 
35J96, 
35Q60, 
65N21, 
65N35, 
\end{AMS}

\pagestyle{myheadings}
\thispagestyle{plain}

\markboth{BRIX, HAFIZOGULLARI, AND PLATEN}{SOLVING THE INVERSE REFLECTOR PROBLEM}


\section{Introduction}

Suppose we have a light source with a given directional characteristic and we would like to generate a prescribed illumination pattern on a target, e.g., project a logo onto a wall.
The classical approach using an aperture has the disadvantage that a part of the light hits the aperture and is lost for target illumination. This efficiency reduction also exists for slide projectors, where some light is absorbed by the film.

In order to use the entire luminous flux emitted by the light source to illuminate the target an optical system comprised of one or more free-form reflectors or free-form lenses can be used instead of an aperture. Then the loss of light can be neglected and the desired illumination pattern is encoded in the shape of the optical surfaces, that in general is unknown.
In the following we will mainly focus on this kind of problem with only one mirror, i.e., we shall determine one desired reflecting free-form surface. This nonlinear inverse problem is called the \emph{inverse reflector problem}, which is a well-known problem in nonimaging optics~\cite{Chaves2008,WMB2005}.

Already 2000 years ago reflectors projecting images, called Chinese magic mirrors, have been hand-crafted of bronze in China and Japan, but the recipe has been lost and reconstructed several times over the ages; see~\cite{Berry2006} and~\cite{MY2001}.
Today such free-form optics are important in illumination applications.
For example they are used in automotive industry for the construction of headlights that use the full light emitted by the lamp to illuminate the road but at the same time do not glare oncoming traffic; see, e.g., \cite{ZNC2011}.
Some other applications include homogeneous illumination for machine-vision purposes or the realization of prescribed patterns in architecture illumination.

However, the solution of the inverse reflector problem is anything but trivial. There are many different schemes to determine a desired reflector like the method of supporting ellipsoids~\cite{KO1997,KO1998} and trial and error approaches~\cite{ASG2008,FDL2010,MMP2009,Savage2007}. The biggest problem of these methods is the computational effort needed to accurately compute reflectors for complex desired illumination patterns on the target. In the beginning of the 21st century it turned out that solving a corresponding partial differential equation (PDE) is a high potential approach for this problem as already indicated in~\cite{RM2002} in 2002. This equation is of \emph{Monge-Amp\`ere type}, which is a family of strongly nonlinear second order PDEs; see~\cite{Gutierrez2001} for details on the theory of Monge-Amp\`ere equations. A first equation of this family was presented by Gaspard Monge at the beginning of the 19th century in his work ``M\'emoire sur la th\'eorie des d\'eblais et des remblais''~\cite{Monge1781} and studied later again by Andr\'e-Marie Amp\`ere in 1820~\cite{Ampere1820}. Equations of this type often arise in the context of \emph{optimal transport problems}, where the task is to find the optimal way to transport excavated material (French: \emph{d\'eblai}), e.g., sand, to piles (French: \emph{remblai}) without losing any of the total mass.
These problems often have an economic background: If one minimizes the cost of transport, which is measured by a quadratic cost functional, under the constraint of \emph{total mass conservation}, one obtains a Monge-Amp\`ere equation.~\cite{Villani2003}
In this spirit, the inverse reflector problem deals with the transportation of light under the constraint that the total light flux emitted by the source is redistributed to the target surface.

The rest of this paper is arranged as follows:
In Section~\ref{sec:state_of_the_art} we review approaches to solve the inverse reflector problem.
Section~\ref{sec:math_formulation} focuses on the modeling of the inverse reflector problem via an equation of Monge-Amp\`ere type.
Our new approach for numerically solving equations of Monge-Amp\`ere type is detailed in Section~\ref{sec:collocation_method}.
In particular, we explain how to handle the boundary conditions arising in the inverse reflector problem.
Numerical results for benchmark problems for the Monge-Amp\`ere equations and for the inverse reflector problem are presented in Section~\ref{sec:results}.
The paper closes with the conclusion and an outlook in Section~\ref{sec:conclusionoutlook}.


\section{State of the art of the solution of the inverse reflector problem}\label{sec:state_of_the_art}

There are plenty of existing approaches for solving the inverse reflector problem; for a detailed overview we refer the reader to~\cite{PP2005}.
Most of those methods can be grouped into three classes, namely \emph{brute-force approaches}, \emph{methods of supporting ellipsoids}, and \emph{Monge-Amp\`ere approaches}. We give an overview of these methods in Sections~\ref{subsec:brute_force},~\ref{subsec:ellipsoid_method}, and~\ref{subsec:ma_methods}, respectively. Other techniques which do not fit into these three classes are discussed in Section~\ref{subsec:other_methods}.


\subsection{Brute-force approaches}\label{subsec:brute_force}

At the beginning of the 21st century many trial and error methods have been developed. The idea of these iterative schemes is as follows: For an initial reflector in optical setup the resulting illumination pattern on the target area is computed using a ray tracing software. Then the illumination pattern is compared with the desired one, where a typical measure of the error is the deviation at previously selected test points in the Euclidean norm.
Afterwards the reflector surface is slightly perturbed using ideas from optimization (e.g., simulated annealing) and the setup with the new reflector is simulated again. If the error has decreased, the new reflector is used as initial guess for the next iteration and the procedure is repeated; see, e.g., \cite{ASG2008,FDL2010,MMP2009,Savage2007}.

The advantage of these methods is, that there are only very few restrictions for the optical setting. For example even extended light sources and mirrors with different reflectivities can be considered.
However, the main drawback is the computing time required, because repeated simulations of the setup using costly ray tracing techniques are needed.

Another approach is presented by Weyrich, Peers, Matusik and Rusinkiewicz~\cite{WPM+2009}, where the mirror is assumed to be comprised of many small facets. Each facet is directed to a prescribed position on the target area. This composite reflector in general has a discontinuous surface, which causes artifacts and is not favorable for production. Therefore, in order to smoothen the solution in a post-processing step, the facets are sorted according to their slopes and adjusted with respect to their heights. Nevertheless, the results have relatively low quality and still show many artifacts.


\subsection{Methods of supporting ellipsoids}\label{subsec:ellipsoid_method}

An ellipse in the plane in general has two foci with the property that light emitted at one focus and reflected at the interior of the ellipse is focused at the second focus. Ellipsoids of revolution in three dimensions have the same property. Kochengin and Oliker~\cite{KO1997,KO1998} (see also~\cite{Oliker2006}) therefore proposed the following method: For each point on the target which needs to be illuminated one defines an ellipsoid of revolution whose one focus is located at the light source and the other one at the target point. However, for each target point there are infinitely many ellipsoids of revolution with this property only differing in their diameters. Therefore in an iterative process the diameters of the ellipsoids are determined starting from some initial guess for each ellipsoid. In each iteration the reflector is defined as the convex hull of the intersection of the interiors of all ellipsoids.
The result is a reflector whose surface consists of glued surface segments of ellipsoids of revolution. Since the initial reflector in general does not produce the desired illumination on the target in every iteration the diameters of each ellipsoid besides the first one are shrunken until convergence.

This method requires in each step a numerical integration over the emission solid angles of the source and a large number of optimization steps. Therefore the complexity of this scheme quickly grows with the number $K$ of ellipsoids of revolution. Suppose we would like to determine a reflector whose reflection on the target is exact for each target point up to an accuracy $\gamma>0$, then the number of iterations scales like  $\mathcal{O}(\frac{K^4}{\gamma}\log\frac{K^2}{\gamma})$; see~\cite{KO2003}.
Therefore it is difficult to use this method for practical illumination patterns of higher resolution.

For the special case where the target is assumed to be infinitely far away from the light source Caffarelli, Kochengin, and Oliker~\cite{CKO1999} developed a variant, which uses paraboloids of revolution instead of ellipsoids.


\subsection{Other methods}\label{subsec:other_methods}

The simultaneous multiple surfaces (SMS) method developed by Mi\~nano et al.~\cite{MG1992} constructs rotational symmetric optical systems which couple a prescribed set of incoming wave fronts with prescribed conjugate wave fronts. This method was also extended to design optics in non-rotational symmetric cases; see, e.g., \cite{BMB+2004}. Optical systems which are computed using the SMS method can be found in~\cite{BMB+2004a, MBL+2009, MBD+2004}. This method always computes a pair of surfaces. Hence it cannot be applied to solve the inverse reflector problem with just one reflective surface.

Wang~\cite{Wang1996,Wang2004a} shows that the inverse reflector problem is in fact an optimal transportation problem. By taking into account also the dual reflector the problem can be reformulated as a linear optimization problem. Unfortunately, since the number of linear inequality constraints quickly grows with the number of pixels, the complexity for the linear programming is very high and thus this method is not feasible for images of medium or higher resolution.

Another scheme has been developed by Fe{\ss}ler et al.~\cite{KOS+2009,ZFJ+2012}, which computes single refractive or reflective surfaces to produce a prescribed density distribution on a prescribed target. To the best of the authors' knowledge the method is not completely published.


\subsection{Monge-Amp\`ere approaches}\label{subsec:ma_methods}

The inverse reflector problem statement for a point light source can be considered as an optimal transportation problem leading to strongly nonlinear second order PDEs of Monge-Amp\`ere type; see, e.g., \cite{KW2010,Schruben1972}.

Brickell, Marder end Westcott~\cite{BMW1977} already in 1977 started to develop methods to solve the inverse reflector problem based on Monge-Amp\`ere type equations. Later Engl and Neu\-bauer~\cite{EN1991,Neubauer1997} investigated a conjugate gradient method with certain constraints to solve this problem via a Monge-Amp\`ere type equation. Ries and Muschaweck~\cite{RM2002} also developed a method based on this type of equations. However, this numerical method has only been published very fragmentarily to the best of the authors' knowledge.

A compromise between a trial and error method and the solution of a PDE is proposed by Fournier, Cassarly, and Rolland; see~\cite{Cassarly2010,FCR2008}.
While the case of extended light sources is considered, which poses many additional problems, the discussion is restricted to the special case of rotationally symmetric reflectors.
In an iteration at first the reflector for a point light source is computed by solving an ordinary differential equation. Then the resulting surface is tested using a ray tracing software in a setup with an extended light source. If the result is not good enough, the differential equation for a point light source is solved again for an adjusted target illumination pattern, where the modification comes from the difference between the simulated and the desired target illumination. This procedure is repeated until a stopping criterion holds true.

In a recent preprint Prins et al.~\cite{PTR+2013} derive an equation of Monge-Amp\`ere type for the inverse reflector problem and provide a numerical method for its solution only for a light source that produces parallel light beams, i.e., they aim for the special case of the far field where the light source is assumed to be infinitely far away from the reflector.
In this particular case the Monge-Amp\`ere type equation reduces to a simpler Monge-Amp\`ere equation, which is called a Monge-Amp\`ere equation of standard type and is easier to solve; see Section~\ref{sec:collocation_method} for a detailed discussion.


\section{Mathematical formulation of the inverse reflector problem}\label{sec:math_formulation}

In Section~\ref{sec:state_of_the_art} we saw that there are many different approaches to solve the inverse reflector problem. As proposed in the approaches discussed in Section~\ref{subsec:ma_methods} we follow the strategy of first modeling the inverse reflector problem using a PDE of Monge-Amp\`ere type and solving this equation in a second step. We now therefore turn to a formulation of the problem in mathematical terms. There exist different approaches to deduce an equation of Monge-Amp\`ere type for this kind of problem~\cite{BMW1977,EN1991,KW2010,Schruben1972}.
We choose the formulation from Karakhanyan and Wang~\cite[Proposition 2.2]{KW2010}, which is based on an energy conservation equation. The advantage of this approach is that for the special case in which the light source is located in-plane with the target, we obtain a relatively simple Monge-Amp\`ere equation of standard type. In the following we use the notation and results given by Karakhanyan and Wang~\cite{KW2010}.

In Subsection~\ref{subsec:problem_statement} we first define the problem statement. Then a corresponding equation of Monge-Amp\`ere type is set up in Subsection~\ref{subsec:ma_equation} and a result and some remarks on the existence and uniqueness of the solution of this problem are discussed in Subsection~\ref{subsec:reflector_existence_uniqueness}.


\subsection{Problem statement}\label{subsec:problem_statement}

Let us now fix the mathematical description of the problem.

\begin{ownproblem}{IR}\label{problem:IR}
  (see, e.g., the introduction in~\cite{KW2010})\\
  Let a point light source be given, which emits light in all directions given by the set $U\subset S_2:=\{\vec{X}\in\R^3\,:\,\norm{\vec{X}}_2=1\}$. The luminous intensity of the source is modeled by the density function $f:U\rightarrow \R^+:=\{x\in\R\,:\,x>0\}$. Furthermore we have a target area $\Sigma$ given in implicit form as $\Sigma:=\{\vec{Z}\in\R^3\,:\,\psi(\vec{Z})=0\}$ for an appropriate function $\psi$. Let $g:\Sigma\rightarrow \R^+$ be a prescribed density function on the target area $\Sigma$.

  Find a reflector $\Gamma$ which redistributes the entire light emitted from the light source, such that the given target illumination defined by $g$ is generated; see Figure~\ref{fig:reflector_parametrization}.
\end{ownproblem}

\begin{figure}[ht]
  \centering
  \includegraphics{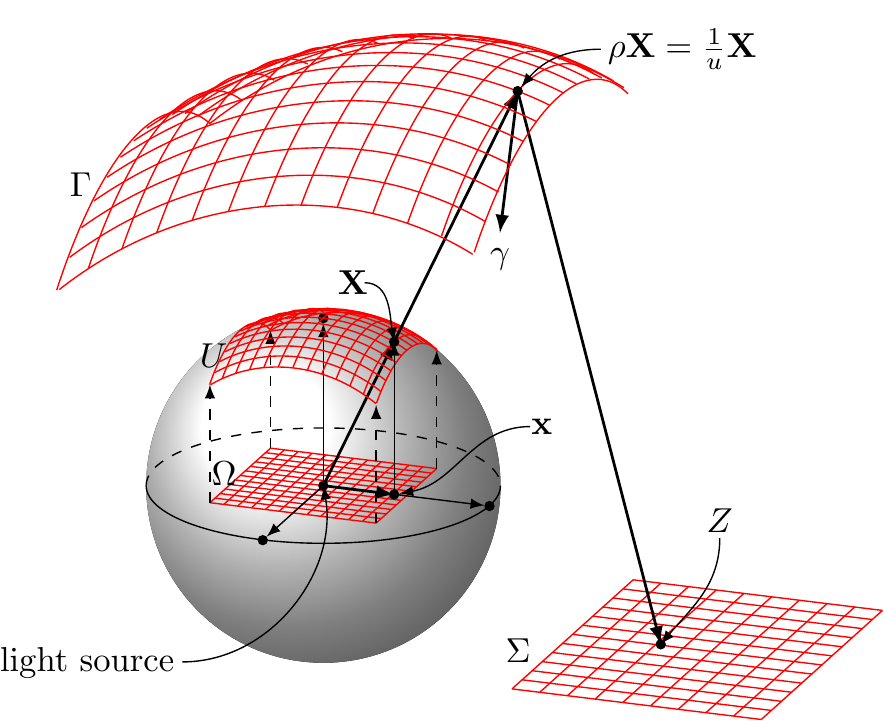}
  \caption{Sketch of the point light source, reflector $\Gamma$, and target $\Sigma$.}
  \label{fig:reflector_parametrization}
\end{figure}


\subsection{Monge-Amp\`ere equation}\label{subsec:ma_equation}

We now aim for deriving the Monge-Amp\`ere type equation corresponding to Problem~\ref{problem:IR}. Therefore we first parameterize the desired reflector by $\Gamma:=\{\rho(\vec{X})\vec{X}\,:\,\vec{X}\in U\}$, where $\rho:U\rightarrow\R^+$ is an appropriate distance function. As already indicated in Figure~\ref{fig:reflector_parametrization} we assume that $U$ is located in the northern hemisphere of $S_2$, such that each $\vec{X}\in U$ can be expressed by $\vec{X}=(\vec{x},\omega(\vec{x}))^T$ where $\vec{x}\in\R^2$ and $\omega:=\omega(\vec{x}):=\sqrt{1-\norm{\vec{x}}_2^2}$.
Therefore we can use the domain $\Omega:=\{\vec{x}\in\R^2\,:\,(\vec{x},\omega(\vec{x}))^T\in U\}$ in $\R^2$ to parameterize the set $U \subset \R^3$.

Under these assumptions Karakhanyan and Wang~\cite[Proposition 2.2]{KW2010} give a PDE for Problem~\ref{problem:IR} for the desired function $\rho$. However, substituting $\rho$ by $\frac{1}{u}$ results in an equation with slightly simpler expressions (see~\cite[Remark 2.1]{KW2010}), which leads to the following result.

\begin{theorem}\label{theorem:MA_reflector}
  (see~\cite[Remark 2.1 and Proposition 2.2]{KW2010})\\
  Let the density functions $f$ and $g$ be given, satisfying the \emph{energy conservation}
  \begin{align}\label{eq:energy_conservation}
    \int_{U} f\,\dd S= \int_{\Sigma} g \, \dd S.
  \end{align}
  Moreover, let the shape of $\Sigma$ be defined by the function $\psi$ as in Problem~\ref{problem:IR}.
  Define
  \begin{align*}
    \omega &:= \sqrt{1-\norm{\vec{x}}_2^2},
      & t &:= 1-u\frac{z_3}{\omega},\\
    \tilde{a} &:= \norm{\DD u}_2^2 - (u- \DD u^T \vec{x} )^2,
      & \hat{\DD}u &:= (\DD u,0)^T,\\
    \tilde{b} &:= \norm{\DD u}_2^2 +u^2 -(\DD u^T \vec{x})^2,
      & \mathcal{N} &:= I+\frac{\vec{x}\vec{x}^T}{\omega^2},
  \end{align*}
  and
  \begin{align*}
    \vec{X} &= (\vec{x},\omega),
      & \vec{Z}_0 &:= \frac{2}{\tilde{a}} \hat{\DD}u,
      & \vec{Z} &= \frac{1}{u}\vec{X} + t\left(\vec{Z}_0-\frac{1}{u}\vec{X}\right),
  \end{align*}
  where $\vec{X}=(x_1,x_2,x_3)^T\in U$ and $\vec{Z}=(z_1,z_2,z_3)^T\in\Sigma$.
  We assume that $t>0$, i.e., $\frac{\omega}{u}>z_3$, and $\nabla\psi^T(\vec{Z}-\frac{1}{u}\vec{X})>0$.
  Then the inverse reflector problem states:

  Find a function $u:\Omega\rightarrow \R$, such that
  \begin{align}
   \label{eq:MA_reflector}
    &\det\left(\textnormal{D}^2 u+\frac{\tilde{a}z_3}{2tx_3}\mathcal{N}\right)
      = -\frac{(u\vec{Z}_0-\vec{X})\cdot\nabla \psi}{t^2\norm{\nabla\psi}_2\omega}\cdot
        \frac{\tilde{a}^3}{4\tilde{b}}\cdot\frac{f(\vec{X})}{\omega g(\vec{Z})} &&\text{in }\Omega\text{ and}\\
   \label{eq:MA_reflector_second_boundary}
    &T:U\,\,\rightarrow\,\,\Sigma,\quad
    \vec{X}\,\,\mapsto \,\vec{Z}
    &&\text{is surjective.}
  \end{align}
\end{theorem}

Let us note some comments on the above theorem.

\begin{remark}
  \begin{enumerate}
    \item
      The assumption $\frac{\omega}{u}>z_3$ means, that the third component of the vector that corresponds to the position where a light ray hits the reflector is larger than the third component of the point where this light ray hits the target $\Sigma$ after reflection. In other words, the light ray must be directed downwards after the reflection.
    \item
      The second condition $\nabla\psi^T(\vec{Z}-\frac{1}{u}\vec{X})>0$ is needed for a technical reason and just fixes the direction of the normal on $\Sigma$. If this constraint is not fulfilled one can simply replace $\psi(\vec{Z})$ by $-\psi(\vec{Z})$ for all $\vec{Z}\in\R^3$.
    \item
      If one sets $z_3\equiv 0$, the equation \eqref{eq:MA_reflector} simplifies to
      \begin{align}\label{eq:MA_reflector_special}
        \det\left(\textnormal{D}^2 u\right)
            &= -\frac{(u\vec{Z}_0-\vec{X})\cdot\nabla \psi}{\norm{\nabla\psi}_2\omega}\cdot
              \frac{\tilde{a}^3}{4\tilde{b}}\cdot\frac{f(\vec{X})}{\omega g(\vec{Z}_0)},
      \end{align}
      which is a \emph{Monge-Amp\`ere equation of standard type}, i.e., the left hand side is only the determinant of the Hessian of $u$; see~\cite[Remark 2.1]{KW2010}.
  \end{enumerate}
\end{remark}

\subsection{Existence and uniqueness}\label{subsec:reflector_existence_uniqueness}

The existence of solutions of Problem~\ref{problem:IR} is ensured by the following result.

\begin{theorem}
    (see~\cite[Theorem A]{KW2010})\\
    Suppose we have two functions $f$ and $g$ which fulfill the energy conservation condition \eqref{eq:energy_conservation} as in Theorem~\ref{theorem:MA_reflector}. Let moreover $\vec{p}$ be an element of the light cone $\mathcal{C}_U$ of the light source, i.e.,
    \begin{align*}
      \vec{p}\in\mathcal{C}_U:=\{\vec{p}\in\R^3\,:\,\frac{\vec{p}}{\norm{\vec{p}}_2}\in U\},
    \end{align*}
    and one of the conditions
    \begin{enumerate}
      \item $\norm{\vec{p}}_2>2\sup_{\vec{q}\in\Sigma}\norm{\vec{q}}_2$ or
      \item $\Sigma\subset\mathcal{C}_V:=\{t \vec{X}\,:\, t>0, \vec{X}\in V\}$ for a region $V\subset S^2$ with $\bar{U}\cap\bar{V}=\emptyset$
    \end{enumerate}
    is satisfied.
    Then there exists a reflecting surface, which is the solution of the inverse reflector problem~\ref{problem:IR} and contains the point $\vec{p}$.
\end{theorem}

We first notice that the solution is not unique. If we have a solution for one $\vec{p}\in\mathcal{C}_U$, we know that there exist other solutions for $c\vec{p}$ for each $c> 1$.

Even if we fix $\vec{p}\in\mathcal{C}_U$, there are at least two solutions that contain $\vec{p}$.
A solution is called $R$-convex if it fulfills an \emph{ellipticity constraint}, i.e., the matrix $\textnormal{D}^2 u+\frac{\tilde{a}z_3}{2tx_3}\mathcal{N}$ in \eqref{eq:MA_reflector} is required to be positive definite.
Moreover, a solution is defined to be $R$-concave, if this matrix is negative definite.
In~\cite[Section 7]{KW2010} one can find a sketch of a proof that for a fixed $\vec{p}\in\mathcal{C}_U$ there are exactly one $R$-convex and one $R$-concave solution. Therefore we need at least to fix the size of the reflector by a point $\vec{p}$ and search for either a $R$-convex or a $R$-concave solutions to obtain uniqueness. This is a necessary condition to ensure well-posedness of the problem.


\section{Solving the Monge-Amp\`ere equations}\label{sec:collocation_method}

We now discuss methods to numerically solve equations of Monge-Amp\`ere type, which is particularly difficult due to the strong nonlinearity of this type of equations.

After an overview of numerical solvers for Monge-Amp\`ere type equations in Subsection~\ref{subsec:pde_solver_state_of_the_art} we discuss our approach in Subsection~\ref{subsec:spline_collocation}.
In order to improve the convergence properties and to speed up the solution process, a multilevel technique will be introduced in Subsection~\ref{subsec:splines_nested}.
Since the existence and uniqueness of a solution is guaranteed, if the Monge-Amp\`ere type equation fulfills an ellipticity condition (see, e.g., \cite[Theorem 1.1 and Remarks (i)]{Urbas1998}), we ensure that this condition holds true by adding a convexity constraint to the equation as detailed in Subsection~\ref{subsec:convexity_constraint}. Since the boundary condition \eqref{eq:MA_reflector_second_boundary} cannot be considered directly, the section closes with the presentation of a technique to realize this type of boundary conditions in Subsection~\ref{subsec:reflector_boundary}.


\subsection{State of the art}\label{subsec:pde_solver_state_of_the_art}

There are several other numerical methods for the solution of Monge-Amp\`ere type equations known, but most of them have clear limitations. Usually the algorithms are only designed to handle boundary value problems for the \emph{Monge-Amp\`ere equation of standard type}
\begin{align}\label{eq:MA_standard}
\det(\DD^2 u(\vec{x}))=f(\vec{x})
\end{align}
for any $\vec{x}\in\Omega\subset\R^2$ with Dirichlet boundary conditions $u(\vec{x})=g(\vec{x})$ for $\vec{x}\in\partial\Omega$.
In particular, in this case the left-hand side determinant may only depend on the Hessian of the solution, but no perturbations in the determinant like $\det(\DD^2 u+A)=f$ for a matrix $A$ are permitted. Numerical methods for this Monge-Amp\`ere equation of standard type can for example be found in~\cite{Awanou2010,BFO2010,BGN+2011,DG2003,DG2004,FN2009,Oberman2008}.

In order to solve the inverse reflector problem we search for a solution of the Monge-Amp\`ere equation given in Theorem~\ref{theorem:MA_reflector}. Since equation \eqref{eq:MA_reflector} is strongly nonlinear with rather cumbersome terms it is very difficult to analyze this equation particularly with regard to weak formulations. To the best knowledge of the authors, there is no closed theory on weak formulations in the classical sense for equations of Monge-Amp\`ere type available.
Hence, Feng and Neilan~\cite{FN2009,FN2009a,FN2011} introduce a new type of weak solution called the moment solution and investigate the following ansatz. They embed strongly nonlinear PDEs into linear PDEs of higher order and study the limit of the vanishing highest order term. For example, the Monge-Amp\`ere equation \eqref{eq:MA_standard} is embedded into an quasilinear elliptic PDE of fourth-order with highest order term $\varepsilon\Delta^2 u$ with $\varepsilon>0$, where the limit $\varepsilon\rightarrow 0$ is studied. Such a scheme is called vanishing moment method; for details we refer to~\cite{FN2009,FN2009a,FN2011}.

Another interesting approach has recently been published by Brenner et al.~\cite{BGN+2011} and by Brenner and Neilan~\cite{BN2012}. They propose a finite element method that leads to a sophisticated consistent discretization and show that the corresponding discrete linearized problem is stable. The main idea is to employ standard continuous Lagrange finite elements instead of finite elements with higher smoothness and to use penalty terms, like those applied in Discontinuous Galerkin methods, to demand for regularity of the solution across interfaces.

In view of practical applications like the inverse reflector problem, it is necessary that a numerical solver efficiently treats Monge-Amp\`ere equations of standard type as well as perturbed equations with Neumann boundary conditions.
At least one of the two methods given by Benamou, Froese, and Oberman~\cite{BFO2010} supports Neumann boundary conditions but is only suited to solve Monge-Amp\`ere equations of standard type. In a subsequent work Froese~\cite{Froese2012} presents a method for Monge-Amp\`ere type equations arising in optimal transportation problems. She uses a Neumann boundary condition and the right-hand side of the Monge-Amp\`ere equation of standard type \eqref{eq:MA_standard} is allowed to also depend on the gradient of $u$. The method by Brenner et al.~\cite{BGN+2011} also permits this dependency on the right-hand side.

A detailed overview of numerical methods for fully nonlinear second order PDEs including methods for Monge-Amp\`ere type equations can be found in the review article~\cite{FGN2013} by Feng, Glowinski, and Neilan.


\subsection{Spline collocation method}\label{subsec:spline_collocation}

We now explore a different simple but very flexible approach for the solution of the Monge-Amp\`ere type equations: The \emph{collocation method} can directly be applied to the strong formulation such as \eqref{eq:MA_reflector}.

The solution is approximated in a finite-dimensional trial space, in our case we choose the space of spline functions because of its good approximation properties. The spline space is spanned by B-spline functions which form an advantageous basis due to its flexible manageability. Moreover this basis is well-known to be numerically very stable and the functions are of minimal support, which favors sparsity in the collocation matrices.

The rest of this subsection is arranged as follows: First we formulate the collocation method in Subsection~\ref{subsec:collocation}. The trial space and a modified B-spline basis that is suited for our particular choice of the collocation points are set up in one spatial dimension in Subsections~\ref{subsec:b-splines} and~\ref{subsec:modified_b-splines}. Finally, the trial space is extended to the two-dimensional case in Subsection~\ref{subsec:tensor_product} via a tensor construction.

\subsubsection{Collocation}\label{subsec:collocation}

Let us now formulate the collocation method for a general second order PDE and therefore introduce some notation. Let $\Omega := (a,b)\times (c,d)\subset\R^2$ be a rectangular domain and let the boundary value problem be given as
\begin{align}\label{eq:general_pde}
  F(x,y,u(x,y),\DD u,\DD^2 u) &= 0, && \text{for }(x,y)^T\in\Omega,\\
  \label{eq:general_boundary}
  G(x,y,u(x,y),\DD u) &= 0, && \text{for }(x,y)^T\in\partial\Omega.
\end{align}
We approximate the exact solution in a finite-dimensional trial subspace, say of dimension $n\in\N$, spanned by some basis functions $B_{1},...,B_{n}\in C^2(\Omega)$. Then the approximate solution is written as
$\hat{u}(x,y) := \sum_{i=1}^{n}c_{i}B_{i}(x,y)$,
where $(x,y)^T\in \bar{\Omega}$ and $c_1,...,c_n$ are the desired coefficients.

Of course we cannot expect such a discrete solution to fulfill the PDE exactly in the whole domain $\Omega$. The idea of the collocation method is that the PDE should be fulfilled pointwise at certain \emph{collocation points}.
Therefore we choose $n$ appropriate pairwise different collocation points, i.e., we select two finite and nonempty subsets $\hat{\Omega}\subset\Omega$ and $\partial\hat{\Omega}\subset \partial\Omega$ with cardinality $|\hat{\Omega}\cup\partial\hat{\Omega}|=n$. Our problem \eqref{eq:general_pde}, \eqref{eq:general_boundary} is then required to be fulfilled exactly at these points, i.e., we end up with a discrete problem which is the nonlinear system of equations
\begin{align}\label{eq:general_pde_discrete}
  F(\tau,\mu,\hat{u}(\tau,\mu),\DD \hat{u},\DD^2 \hat{u}) &= 0, && \text{for }(\tau,\mu)^T\in\hat{\Omega},\\
  \label{eq:general_boundary_discrete}
  G(\tau,\mu,\hat{u}(\tau,\mu),\DD \hat{u}) &= 0, && \text{for }(\tau,\mu)^T\in\partial\hat{\Omega}.
\end{align}

Now we can use some \emph{Newton-type method} to solve \eqref{eq:general_pde_discrete}. For reasons of better stability we favor the application of the \emph{double-dogleg method} (see, e.g., \cite{DS1983}), which is a trust region algorithm of \emph{quasi-Newton type}. We use the variant proposed by Dennis and Mei~\cite{DM1979}, where we invoke an algorithm by Nielsen~\cite{Nielsen1999} for the choice of the trust region radius after each iteration step. Adequate stopping criteria for the iteration process can be found in~\cite{MNT2004}.

Next we will discuss the choice of appropriate basis function and collocation points.

\subsubsection{B-splines}\label{subsec:b-splines}

Let us briefly recall the definition of \emph{B-splines} on the real line, which are the fundament of our basis functions, using the notation of~\cite{Dahmen1998}.

\begin{definition}
Let $[a,b]\subset\R$ be a given interval that serves as the domain of the basis functions. For B-splines of order $n\in\N$ we define for a fixed $N\in\N$ a strictly increasing \emph{knot sequence} $T=\{t_i\}_{i=1}^{N+n}\subset \R$ with
\begin{align}\label{eq:strictly_increasing_knot_sequence}
  t_1<...<t_{n}:=a<...<t_{N+1}:=b<...<t_{N+n}.
\end{align}
For $i=1,...,N$ the $i$-th B-spline $N_{i,n}$ of order $n$ is then defined by the \emph{Cox-de Boor recursion formula}
\begin{align*}
  N_{i,1}(x) &:= \chi_{[t_i,t_{i+1})}(x):=
    \begin{cases}
      1,&\text{if }x\in [t_i,t_{i+1}),\\
      0,&\text{otherwise,}
    \end{cases}
             && \text{if }n=1\text{ and}\\
  N_{i,n}(x) &:= \frac{x-t_i}{t_{i+n-1}-t_i}N_{i,n-1}(x)
                +\frac{t_{i+n}-x}{t_{i+n}-t_{i+1}}N_{i+1,n-1}(x)
   && \text{if }n\geq 2.
\end{align*}
\end{definition}

\begin{remark}
The condition of a strictly increasing knot sequence \eqref{eq:strictly_increasing_knot_sequence}
can be relaxed using de l'H\^opital's rule and \emph{multiple knots} can be permitted; see, e.g., \cite[Chapter IX]{DeBoor1978} or~\cite[Section 2.2]{PT1997}.
\end{remark}

In the following we fix the outer knots in \eqref{eq:strictly_increasing_knot_sequence} at the boundary, i.e., we set $t_1:=...:=t_{n-1}:=a$ and $t_{N+2}:=...:=t_{N+n}:=b$. Multiple knots for the interior knots $t_n,...,t_{N+1}$ result in less smooth B-splines. Since we want to solve a PDE of second order, we need basis functions which are at least twice differentiable. To have minimal computational efforts while fulfilling this constraint we choose the lowest possible order, which is $n=4$, i.e., \emph{cubic splines}, and avoid multiple knots inside $(a,b)$.

\begin{remark}\label{rem:BSplinesVanishingDerivatives}
Since the left outer knots all coincide with the left boundary $a$, $N_{1,n}$ is the only B-spline with a non-zero function value at the left boundary point $a$. Moreover, the first derivative of $N_{j,n}$ vanishes at $a$ for all $j>2$ and the second derivative of $N_{j,n}$ vanishes at $a$ for all $j>3$.
By symmetry, the same holds for the B-splines at the right boundary point $b$.
\end{remark}

\subsubsection{Collocation points and modification of the B-spline basis}\label{subsec:modified_b-splines}

In order to uniquely define a spline from our spline space it suffices to set $N$ linear independent conditions, e.g., to prescribe the function values at $N$ appropriate different nodes. The knots themselves are possibly a good choice for these nodes. However, we only have $N-2$ knots in $[a,b]$, that is there are two open degrees of freedom left.
There are different ways to handle open degrees of freedom, e.g., by setting a not-a-knot condition~\cite[Chapter IV]{DeBoor1978}.
Another possibility is to modify the trial space by lowering the dimension of the spline space, which we will discuss next.

For our purpose it is crucial to keep the approximation properties of the trial space. Therefore it is necessary that each (Taylor-)polynomial of degree $\leq n-1$ can still be reproduced and consequently at least $n$ basis functions must be supported in each subinterval. With regard to the boundary conditions and the clearness of the construction of the modified basis, we build new basis functions at the boundary of the interval $[a,b]$. Without loss of generality we can restrict ourselves to the left boundary case, the right boundary case is handled analogously by symmetry.

In order to modify as few basis functions as possible we choose four new basis functions from the span of the first five B-splines, i.e., we write
\begin{align*}
  B_{i,4}&:=
  \begin{cases}
    \sum_{j=1}^5a_{i,j}N_{j,4}     & \text{for }i=1,...,4,\\
    N_{i+1,4}                      & \text{for }i=5,...,N-6,\\
    \sum_{j=N-4}^{N}a_{i,j}N_{j,4} & \text{for }i=N-5,...,N-2
  \end{cases}
\end{align*}
for some coefficients $a_{i,j}\in\R$, which still have to be determined. Note that this approach of gluing B-splines at the boundary is similar to the procedure used in the construction of WEB-splines~\cite{HRW2001} where some B-splines are glued to improve the numerical stability of the basis, i.e., to lower its condition number.

From \emph{Marsden's identity} we can derive
\begin{align}\label{eq:monomials} x^m=\sum_{j=1}^{N}\frac{(-1)^{n-1-m}\psi_{j,n}^{(n-1-m)}(0)}{(n-1)\cdot...\cdot(m+1)}N_{j,n}(x)
  \quad\text{with}\quad
  \psi_{j,n}(y):=\prod_{l=1}^{n-1}(t_{j+l}-y)
\end{align}
for $m=0,...,n-1$; see, e.g., \cite[(2.3.3)]{Dahmen1998}.

Furthermore the coefficient matrix $A :=[a_{i,j}]_{i=1, \dots, 4,\, j=1, \dots ,5}$ has to be chosen in such a way, that for each $m\in\{0,1,2,3\}$ there exist coefficients $\{c_{i,m}\}_{i=1}^{N-2}$ so that
\begin{align}
  \notag
  x^m &= \sum_{i=1}^{N}c_{i,m}B_{i,4}(x)\\
  \label{eq:spline_old}
  &=
  \sum_{j=1}^5\left(\sum_{i=1}^{4}c_{i,m}a_{i,j}\right)N_{j,4}+
  \sum_{i=6}^{N-5}c_{i,m}N_{i+1,4}+
  \sum_{j=N-4}^{N}\left(\sum_{i=N-5}^{N-2}c_{i,m}a_{i,j}\right)N_{j,4}
\end{align}
holds.
Setting $C:=[c_{i,m}]_{1\leq i\leq 4,\, 0\leq m\leq 3}$ and
equating the coefficients of $N_{j,4}$ for $j=1,...,5$ in \eqref{eq:monomials} and \eqref{eq:spline_old} yields the system of linear equations
\begin{align}\label{eq:linear_equations}
  A^T C = B
\end{align}
where $B:=[b_{j,m}]_{1\leq j\leq 5,\, 0\leq m\leq 3}$ with
$b_{j,m}:=\frac{(-1)^{n-1-m}\psi_{j,n}^{(n-1-m)}(0)}{(n-1)\cdot...\cdot(m+1)}$.
Now we have to choose the matrix $A$, such that \eqref{eq:linear_equations} has a solution matrix $C$.

The trivial solution is $A:=B^T$ and $C:=I$, where $I \in \R^{4 \times 4}$ is the identity matrix. In other words the new basis elements then are exactly the monomials $1$, $x$, $x^2$, and $x^3$ at the first subinterval. The disadvantage of this solution is, that the new basis functions are not shift invariant, i.e., they depend on the location of the knots in the knot sequence $T$. Therefore we will search for a more favorable solution of shift invariant functions.

For this purpose let us choose the ansatz $(BP)P^{-1}=B$ for an invertible matrix $P\in\R^{4\times 4}$ and set $A:=(BP)^T$ and $C:=P^{-1}$. Thus the idea is to define the matrix $A^T$ using appropriate column operations acting on $B$, such that $A^T$ has a preferably simple form.

Of course the choice of $P$ is not unique, such that we can prescribe additional conditions, e.g., in the spirit of the properties of the B-spline functions observed in Remark~\ref{rem:BSplinesVanishingDerivatives}. For Dirichlet boundary conditions it is advantageous, if only one basis function is nonzero at the boundary. The same holds true for Neumann boundary conditions when the first derivative of the basis functions is considered. Moreover, for natural spline interpolation (see, e.g., \cite[Chapter IV]{DeBoor1978}) it is convenient if only one basis function has a non-vanishing second derivative at the boundary. The following choice provides us with a basis which has all these handy properties.

Suppose that we have an equidistant knot sequence with multiple knots at the boundary, i.e., $T=\{t_i\}_{i=1}^{N+4} \subset [a,b] \subset \R$ with
\begin{align}\label{eq:open_knots}
  t_i &:= \begin{cases}
            a, & \text{if }1\leq i\leq 4,\\
            a + \frac{b-a}{N-3}(i-4), & \text{if }5\leq i \leq N,\\
            b, & \text{if }N+1 \leq i \leq N+4.
          \end{cases}
\end{align}
After some elementary computations for this particular case, we can determine an invertible matrix $P$ such that
\begin{align}\label{eq:matrixAnew}
  A = (BP)^T =
  \left[\begin{matrix}
    1          & 0             & 0           & 0           \\[0.3em]
    1          & \frac{1}{4}   & 0           & 0           \\[0.3em]
    1          & \frac{3}{4}   & \frac{1}{2} & 0           \\[0.3em]
    \frac{3}{4}& \frac{15}{16} & \frac{3}{4} & \frac{1}{4} \\[0.3em]
    0          & 0             & 0           & 1
  \end{matrix}\right]^T.
\end{align}
Note that the matrix $A$ is independent of the knots, such that our new basis functions inherit the shift invariance from the B-splines. A plot of the resulting four new basis functions is given in Figure~\ref{fig:modified_bsplines}.

\begin{figure}[ht]
  \centering
  \includegraphics{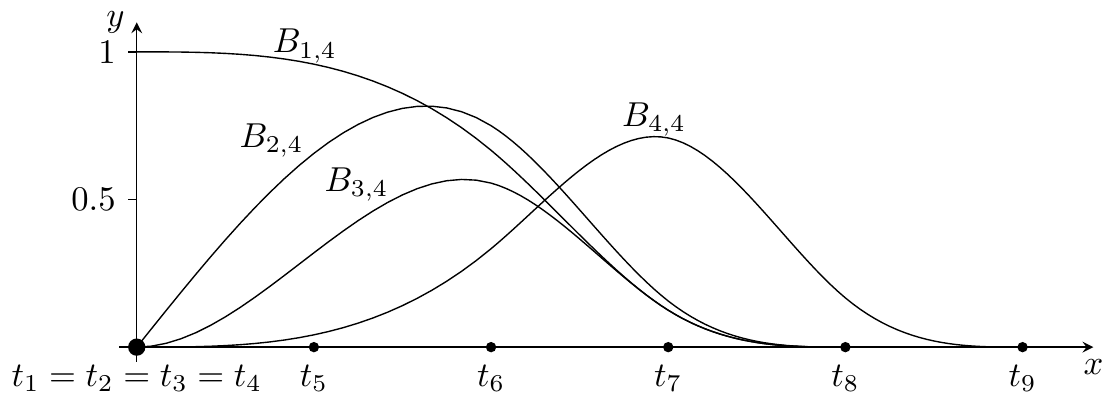}
  \caption{Modified basis functions given by $A$ in \eqref{eq:matrixAnew}.}
  \label{fig:modified_bsplines}
\end{figure}

\begin{remark}
Due to the conditions imposed and the properties pointed out in Remark~\ref{rem:BSplinesVanishingDerivatives}, the matrix $A^T$ in \eqref{eq:matrixAnew} has a lower triangular structure.
\end{remark}

\subsubsection{Tensor-product B-spline basis}\label{subsec:tensor_product}

Since we aim for solving equations of Monge-Amp\`ere type on a two-dimensional rectangle $R=[a,b]\times [c,d]\subset\R^2$, we next define the spline space on $R$ using the usual tensor product construction. Let $T_1$ be the knot sequence as in \eqref{eq:open_knots} for the interval $[a,b]$ and $T_2$ the analog knot sequence for $[c,d]$. As in Section~\ref{subsec:b-splines} for both $T_1$ and $T_2$ we define B-splines of order $n_1$ and $n_2$, respectively.
We then define \emph{tensor-product B-splines} by
\begin{align*}
  N_{i,j;n_1,n_2}(x,y) &:= N_{i,n_1}(x) \, N_{j,n_2}(y)
\end{align*}
for $i=1,...,N_1$, $j=1,...,N_2$ and $(x,y)\in [a,b]\times [c,d]$.

Using the same arguments as in Subsection~\ref{subsec:b-splines} we set $n_1=n_2=4$ and restrict ourselves to knot sequences that are equidistant in the inner of the intervals. As ansatz functions in our collocation method in Subsection~\ref{subsec:collocation} we use the modified tensor-product basis functions
\begin{align*}
  B_{i,j;4}(x,y) &:= B_{i,4}(x) \, B_{j,4}(y)
\end{align*}
for $i=1,...,N_1-2$ and $j=1,...,N_2-2$ and, as in~\cite[Chapter XIII]{DeBoor1978}, the collocation points
\begin{align*}
  \hat{\Omega}
    &:=\{(\tau,\mu)^T\in\Omega\,:\,\tau\in T_1,\,\mu\in T_2\}
      \subset\Omega,\\
  \partial\hat{\Omega}
    &:=\{(\tau,\mu)^T\in\partial\Omega\,:\,\tau\in T_1,\,\mu\in T_2\}
      \subset\partial\Omega,
\end{align*}
which are the knots of the B-splines.


\subsection{Nested iteration}\label{subsec:splines_nested}

Splines are particularly suitable for multilevel strategies, because of the following property: Suppose we have a knot sequence $\hat{T}:=\{\hat{t}_i\}_{i=1}^{M+n} \subset \R$ for the interval $[a,b]\subset\R$ as in \eqref{eq:open_knots} and a second knot sequence $T:=\{t_i\}_{i=1}^{N+n}$ for $[a,b] \subset \R$ that is obtained from $\hat{T}$ by inserting new knots. Then the corresponding spline spaces $S_n(\hat{T})$ and $S_n(T)$ of order $n\in\N$ are nested, i.e., $S_n(\hat{T})$ is a subspace of $S_n(T)$.

Since we use a Newton-type method for solving the discrete nonlinear problem \eqref{eq:general_pde_discrete}, \eqref{eq:general_boundary_discrete} we need an initial guess that is preferably close to the solution. Otherwise solving the problem could be very time consuming or even infeasible, because the Newton-type method might not converge.
We therefore follow an approach based on \emph{nested iteration}: Our computation starts on a very coarse grid. After calculating the solution for this coarse problem, we refine the mesh by knot insertion, e.g., we halve the mesh size in each coordinate direction.
Since the spline spaces of the coarse and refined knot sequences are nested, the coarse solution is also contained in the spline space corresponding to the finer mesh and can be used as an initial guess. We continue with this process until we reach the desired grid resolution.

To be more precise, assume we have $N=N_1=N_2$, i.e., the number of knots is the same in each coordinate direction. Let us now solve the discrete nonlinear problem on a grid with $N^*$ knots in each direction. We start with the coarsest reasonable grid possible in our situation such that the boundary basis functions do not overlap, which is of size $N^0:=11$, and after each iteration we halve the mesh size. Then after nested iteration $k\in\N$ there are $N^k:=2N^{k-1}-1$ knots in the grid in each coordinate direction. If $N^{k+1}\geq N^*$ we interpolate the solution obtained from the $k$-th iteration to the fine grid with $N^*\times N^*$ knots using spline interpolation and solve the final nonlinear problem.


\subsection{Convexity constraint for Monge-Amp\`ere equations}\label{subsec:convexity_constraint}

In order to show the existence and the uniqueness of solutions of Monge-Amp\`ere type equations, it is often required that the equation is elliptic with respect to the solution; see, e.g., \cite[Theorem 1.1 and Remarks (i)]{Urbas1998} or~\cite{TW2009}.

A nonlinear PDE $F[u]=F(\cdot,u,\DD u,\DD^2u)=0$ is said to be \emph{elliptic}, if the matrix
\begin{align}\label{eq:proof_ellipticity}
  \left[\frac{\partial F}{\partial
  r_{ij}}(\gamma)\right]_{i,j}:=\left[\begin{array}{ccc}
  \frac{\partial F}{\partial r_{11}}(\gamma)&...& \frac{\partial F}{\partial r_{1n}}(\gamma)\\
  \vdots&&\vdots\\
  \frac{\partial F}{\partial r_{n1}}(\gamma)&...& \frac{\partial F}{\partial r_{nn}}(\gamma)
\end{array}\right]
\end{align}
is positive definite for all $\gamma=(x,z,p,r)\in V\subset \Omega\times\R\times \R^n\times \tilde{\R}^{n\times n}$, where $\tilde{\R}^{n\times n}$ denotes the space of symmetric $n\times n$ matrices; see, e.g., \cite[Chapter 17]{GT1983}.

In case of a general Monge-Amp\`ere equation
\begin{align}\label{eq:MA_general}
  \det \left(\DD^2 u + A(\cdot,u,\DD u)\right) &= f(\cdot,u,\DD u)
  \qquad\textnormal{ in } \Omega
\end{align}
with appropriate boundary conditions we have the following result.

\begin{lemma}
  Let $\Omega\subset\R^n$ be a domain, $u\in C^2(\Omega)$, and let $A$ be a symmetric $n\times n$ matrix. It is necessary and sufficient to ensure ellipticity for \eqref{eq:MA_general} that
  \begin{align}\label{eq:convexity_general}
    \DD^2 u + A(\cdot,u,\DD u)\text{ is positive definite}.
  \end{align}
  Consequently the right-hand side $f$ has to be positive to allow elliptic solutions.
\end{lemma}
\begin{proof}
  For $A\equiv 0$ this result has been proven, e.g., in~\cite[Lemma 1]{FO2011}. Otherwise, the result is also well-known (see, e.g., \cite{TW2009}), but the authors could not find a proof in the literature. We therefore extend the proof given by Froese and Oberman~\cite[Lemma 1]{FO2011} to this more general result.

  Let $\tilde{\mathcal{W}}$ be the cofactor matrix of the symmetric matrix $\mathcal{W}:=\DD^2u+A$; see, e.g., \cite[Section 4.3]{Strang1980} for a definition.
  As a consequence of Cramer's rule we have  $\mathcal{W}^{-1}\det(\mathcal{W})=\tilde{\mathcal{W}}^T$. It follows that $\tilde{\mathcal{W}}$ is positive definite if and only if $\mathcal{W}$ is positive definite. Therefore we only have to prove that
  \begin{align}\label{eq:proof_determinant}
    \DD_r\det(\mathcal{W})
      :=\left[\frac{\partial\det(\mathcal{W})}{\partial (\DD^2 u)_{i,j}}\right]_{i,j}
      =\tilde{\mathcal{W}}
  \end{align}
  is true, which is an alternative expression of \eqref{eq:proof_ellipticity} for the Monge-Amp\`ere equation \eqref{eq:MA_general}.

  Expanding the determinant along the $i$-th row using Laplace's formula yields
  \begin{align*}
    \det(\mathcal{W})
    = \sum_{j=0}^n \tilde{\mathcal{W}}_{i,j}\mathcal{W}_{i,j}
    = \sum_{j=0}^n \tilde{\mathcal{W}}_{i,j}((\DD^2u)_{i,j}+A_{i,j}).
  \end{align*}
  By definition the cofactor $\tilde{\mathcal{W}}_{i,j}$ is independent of $\mathcal{W}_{i,j}$ and therefore of $(\DD^2 u)_{i,j}$ as well. In addition, the matrix $A$ is independent of $\DD^2 u$, such that we have
  \begin{align*}
    \frac{\partial\det(\mathcal{W})}{\partial (\DD^2 u)_{i,j}}=\tilde{\mathcal{W}}_{i,j}.
  \end{align*}
  We therefore proved that \eqref{eq:proof_ellipticity} for the Monge-Amp\`ere equation \eqref{eq:MA_general} is positive definite if and only if $\DD^2u+A$ is positive definite.
\end{proof}

\begin{remark}
In case that $A\equiv 0$, the Hessian matrix $\textnormal{D}^2u$ must be positive definite to ensure ellipticity for \eqref{eq:MA_general}, which means that $u$ has to be strictly convex. Therefore condition \eqref{eq:convexity_general} can be viewed as some kind of convexity condition.
\end{remark}

As proposed in~\cite{FO2011,FO2011a} we take \eqref{eq:convexity_general} into account as an additional constraint. Note that the positive definiteness of a symmetric matrix from $\R^{2 \times 2}$ is equivalent to a positive determinant and positive diagonal entries. The positivity of the determinant is guaranteed if $f>0$. To forbid solutions with non-positive diagonal entries in the matrix \eqref{eq:convexity_general} we define the modified determinant
\begin{align*}
  {\det}^+(\mathcal{W})
    &:= \max\{0,\mathcal{W}_{1,1}\}\max\{0,\mathcal{W}_{2,2}\} - \mathcal{W}_{1,2}^2,
\end{align*}
which is non-positive if at least one diagonal entry is non-positive and otherwise equals $\det(\mathcal{W})$. This idea is similar to the modified determinant in~\cite{FO2011,FO2011a} and also in~\cite[Section 4.3]{Froese2012}.

To avoid problems, such as non-uniqueness of a solution in case of a singularity, where $\DD^2u+A$ has an eigenvalue equal to zero, we introduce a parameter $\lambda\geq 0$ and subtract a penalty term
\begin{align}\label{eq:modified_det}
  {\det}^+_{\lambda}\mathcal{W}
    &:= {\det}^+\mathcal{W} - \lambda\left[(\min\{0,\mathcal{W}_{1,1}\})^2 + (\min\{0,\mathcal{W}_{2,2}\})^2\right]
\end{align}
to ensure that the modified determinant has a negative value instead of just a non-positive value. This was done similarly by Froese~\cite[Section 4.3]{Froese2012}.
Now we replace the determinant in \eqref{eq:MA_general}, consider the equation
\begin{align}\label{eq:MA_convexity}
  {\det}^+_{\lambda} \left(\DD^2 u + A(\cdot,u,\DD u)\right) = f(\cdot,u,\DD u)
\end{align}
instead, and obtain the following result.

\begin{lemma}
  Let $\Omega\subset\R^2$ be a domain, $u\in C^2(\Omega)$, and $f>0$. If $u$ is a classical solution of the modified Monge-Amp\`ere equation \eqref{eq:MA_convexity}, then $u$ also solves the original equation \eqref{eq:MA_general} and simultaneously fulfills the ellipticity constraint \eqref{eq:convexity_general} and vice versa.
\end{lemma}
\begin{proof}
  Let $u\in C^2(\Omega)$ be a classical solution of \eqref{eq:MA_convexity} and $\mathcal{W}:=\DD^2 u + A(\cdot,u,\DD u)$. Since ${\det}^+_{\lambda}\mathcal{W}=f>0$, the diagonal entries of $\mathcal{W}$ are positive because otherwise we would have ${\det}^+_{\lambda}\mathcal{W}\leq 0 < f$. Therefore $f={\det}^+_{\lambda}\mathcal{W} = \det \mathcal{W}$ holds true. It immediately follows, that $u$ also solves \eqref{eq:MA_general} and \eqref{eq:convexity_general} simultaneously.

  Conversely, if $u$ solves \eqref{eq:MA_general} and \eqref{eq:convexity_general}, the diagonal entries of $\mathcal{W}$ are positive and we have $f=\det \mathcal{W}={\det}^+_{\lambda}\mathcal{W}$.
\end{proof}


\subsection{Boundary conditions for the inverse reflector problem}\label{subsec:reflector_boundary}

Let us now come back to the solution of the inverse reflector problem. The equation of Monge-Amp\`ere type \eqref{eq:MA_reflector} in Theorem~\ref{theorem:MA_reflector} already is of the desired form \eqref{eq:general_pde} to be treated by the collocation method. But this is not true for the boundary condition \eqref{eq:MA_reflector_second_boundary}, because it is a constraint for the desired mapping $T$ on the whole domain $U$ and not only for $T$ restricted to the boundary of $U$, i.e., for $T|_{\partial U}$. Thus condition \eqref{eq:MA_reflector_second_boundary} is not of the general form \eqref{eq:general_boundary}, which makes it difficult to handle. To overcome this problem we first assume that the boundary of $U$ is supposed to be mapped by $T$ onto the boundary of $\Sigma$ and the interior of $U$ to the interior of $\Sigma$. To the best of the authors' knowledge this assumption has not yet been proven to hold for the inverse reflector problem.
But it is worthwhile noting that when the mirror surface is interpreted as an extended light source this assumption corresponds to the \emph{edge ray principle} from nonimaging optics~\cite[Chapter 3]{WMB2005}. In brief, it states that when the rays emitted at the boundary of the light source, the edge rays, are mapped to the boundary of the target it is ensured that all other rays emitted by the light source are also mapped to the target, i.e. energy conservation holds.

However, Froese~\cite{Froese2012} discusses a simpler but related problem where a similar assumption to \eqref{eq:MA_reflector_second_boundary} holds true and she proposes to replace this assumption by a simpler boundary condition. Therefore we follow her strategy proposed in~\cite[Section 3.3]{Froese2012} to render our boundary condition manageable.

The idea is to first replace the constraint \eqref{eq:MA_reflector_second_boundary}, i.e., $T(U)=\Sigma$, by $T(\partial U)=\partial\Sigma$. Since we work on $\Omega$, which is isomorphic to $U$, we write $T(\partial \Omega)=\partial\Sigma$. We then only need to prescribe the normal component of the mapping $T|_{\partial \Omega}$, such that we obtain the boundary condition
\begin{align}\label{eq:bc_unknown_phi}
  T(\cdot,u,\DD u)^T\nu(\cdot)&=\phi(\cdot) \qquad\text{in }\partial\Omega,
\end{align}
where $\nu$ is the outer normal vector of $\partial\Omega$ and the normal component $\phi$ of $T$ is an a priori unknown function.

Unfortunately, we now have a circular dependency problem.
If we knew $\phi$, we would have the boundary condition in the desired form \eqref{eq:general_pde}. Of course if the solution $u^*$ of the problem and therefore the correct mapping $T^*$ is known, then $\phi(\vec{x})=T^*(\vec{x},u^*,\DD u^*)^T\nu(\vec{x})$. But the function $\phi$ is unknown unless the solution of Problem~\ref{problem:IR} is solved, where the boundary condition is needed to identify the solution.

In order to disrupt this circular dependency problem, a \emph{Picard-type iteration} is proposed in~\cite{Froese2012} for a similar but different problem:
We iterate and start with an initial guess $\phi^0$. For $k=1,2,...$ we determine $\phi^k$ by first solving the Monge-Amp\`ere equation \eqref{eq:MA_reflector} with boundary condition \eqref{eq:bc_unknown_phi} using $\phi^{k-1}$ instead of $\phi$. The solution $u^{k-1}$ defines a reflector mapping $T^{k-1}:=T(\cdot,u^{k-1},\DD u^{k-1})$, which not necessarily maps $\partial\Omega$ onto $\partial\Sigma$ but onto $\partial\Sigma^{k-1}$ for the image $\Sigma^{k-1}$ of the mapping $T^{k-1}$, which in general differs from $\Sigma$.
In order to correct this we apply the orthogonal projection of $\partial\Sigma^{k-1}$ onto $\partial\Sigma$ using the standard scalar product and define
\begin{align}\label{eq:phi_k}
  \phi^k(\vec{x})
    := \left[\underset{\vec{z}\in\partial\Sigma}{\arg\min}\,
        \normlr{ \vec{z} - T^{k-1}(\vec{x},u^{k-1},\DD u^{k-1}) }_2^2\right]\nu(\vec{x})
    \qquad\text{for }\vec{x}\in\partial\Omega.
\end{align}
The boundary condition then reads
\begin{align*}
  \left[T(\vec{x},u^k,\DD u^k)
      - \underset{\vec{z}\in\partial\Sigma}{\arg\min}\,
        \normlr{ \vec{z} - T(\vec{x},u^{k-1},\DD u^{k-1}) }_2^2
  \right]^T\nu(\cdot) = 0
    \qquad\text{for }\vec{x}\in\partial\Omega,
\end{align*}
where $u^{k-1}$ is known and $u^k$ is the desired interim solution in step $k$.

\paragraph{Existence}
For the solution of the inverse reflector problem we need to ensure the conservation of energy \eqref{eq:energy_conservation}. In order to obtain the correct $\phi$, we solve the reflector problem for $\phi^k$ and therefore for a different target $\Sigma^k$ for $k=1,2,...$ in the Picard-type iteration. The energy conservation then holds for $\Sigma^k$. Note that the prescribed density function $g$ on the target $\Sigma$ of Problem~\ref{problem:IR} can be continued with zero outside of $\Sigma$. However,
we need to compensate for the difference in energy and ensure the energy conservation condition to hold on $\Sigma$ by scaling the density function $f$ with an appropriate constant $c>0$ defined by
\begin{align*}
  c := \frac{\int_{\Sigma^k} g \, \dd S}{\int_{\Sigma} g \, \dd S};
\end{align*}
see~\cite[Section 3.4]{Froese2012}. Since we do not know $\Sigma^k$, also $c$ is unknown. Thus, we introduce $c$ as a new degree of freedom in our subproblems for the different right-hand sides $\phi^k$ of the boundary condition \eqref{eq:bc_unknown_phi} for $k=1,2,...$ and replace $f$ in the Monge-Amp\`ere equation \eqref{eq:MA_reflector} by $cf$ which guarantees the existence of a solution.

\paragraph{Uniqueness}
However, we cannot expect that there is only one $R$-convex solution, i.e., a solution of the Monge-Amp\`ere equation \eqref{eq:MA_reflector} in the elliptic case, for each inverse reflector problem. In fact there are infinitely many solutions; see Subsection~\ref{subsec:reflector_existence_uniqueness}. One possible choice of a condition to ensure uniqueness is to fix the size of the reflector. The reflector is parameterized by the distance function $u$, which controls the size of its shape. Hence, similar to~\cite[Section 3.4]{Froese2012}, we fix a parameter $\mathcal{G}>0$ and add the constraint
\begin{align}\label{eq:integral_constraint}
  \int_{\Omega} u(\vec{x}) \, \dd\vec{x} = \mathcal{G}.
\end{align}

\paragraph{Resulting problem}
Collecting all the conditions, we obtain the subproblems
\begin{align*}
  {\det}_{\lambda}^+ \left(\DD^2 u^k + A(\vec{x},u^k,\DD u^k)\right) &= cb(\vec{x},u^k,\DD u^k)
    && \text{for }\vec{x}\in\Omega,\\
  T^k(\vec{x})^T \nu(\vec{x}) &= \phi^{k-1}(\vec{x})
    && \text{for }\vec{x}\in\partial\Omega,\\
  \int_{\Omega} u(\vec{x}) \, \dd\vec{x} &= \mathcal{G}
\end{align*}
for $k=1,2,...,k_{\max}$, where $u^k$ and $c$ are the unknowns, $\mathcal{G}>0$ and $\lambda\geq 0$ are fixed parameters, $b$ equals the right-hand side of the Monge-Amp\`ere equation \eqref{eq:MA_reflector}, and $\phi^{k-1}$ is the orthogonal projection as defined in \eqref{eq:phi_k}.

\begin{remark}
  There is not plenty of existence and uniqueness theory available for equations of Monge-Amp\`ere type, in particular not for the general case where the determinant does not only depend on the Hessian of $u$. The most adequate theorems for our situation found by the authors are formulated in~\cite[Theorem 1.1]{LTU1986} and~\cite[Theorem 1.1]{Urbas1998} with Neumann boundary conditions. But unfortunately this cannot be applied because of the missing regularity of the boundary $\partial\Omega$ and the fact that our boundary conditions are nonlinear. However, it is worth noting that both theorems state that under some conditions there exists a unique solution of the elliptic Monge-Amp\`ere type equation, so the ellipticity constraint \eqref{eq:convexity_general} is important.
\end{remark}


\section{Numerical simulations}\label{sec:results}

In this section we present some test cases which we use to verify our numerical solver for the inverse reflector problem. First we consider in Subsection~\ref{subsec:test_cases_standard_ma} five benchmark test cases for the Monge-Amp\`ere equation of standard type \eqref{eq:MA_standard}. Since these have also been discussed by Froese and Oberman~\cite{FO2011,FO2011a}, we can compare the convergence behavior of their and our methods. Afterwards we discuss results for Problem~\ref{problem:IR} in Subsection~\ref{subsec:test_cases_inverse_reflector}.

\paragraph{General implementation remarks}
All the computations have been carried out on a standard personal computer equipped with an AMD Phenom II X4 955 processor running at $3.2\,\rm{GHz}$. In order to verify the numerical solutions of the inverse reflector problem we perform a forward simulation of the reflector using the ray tracing software POV-Ray~\cite{CFK+1991}.

Note that in each step of the nonlinear solver, i.e., in each Newton iteration, in our collocation method we solve a sparse system of linear equations. For this purpose we use the unsymmetric multifrontal sparse LU factorization package (UMFPACK)~\cite{DD1997}.

\subsection{Test cases for the Monge-Amp\`ere equation of standard type}\label{subsec:test_cases_standard_ma}

We first define in Subsection~\ref{subsubsec:test_cases_five} five test cases. The results are given in Subsection~\ref{subsubsec:experiments_standard}.


\subsubsection{Five test cases}\label{subsubsec:test_cases_five}

Let us define $\Omega:=(0,1)\times(0,1)\subset\R^2$, $\vec{x}:=(x,y)^T\in\Omega$, and $\vec{x}_0:=(\frac{1}{2},\frac{1}{2})^T\in\Omega$. We consider five examples for the Monge-Amp\`ere equation of standard type with Dirichlet boundary conditions, i.e.,
\begin{align}\label{eq:example_MA}
  \det\left(\DD^2 u(\vec{x})\right) &= f(\vec{x})
    && \text{for }\vec{x}\in\Omega\text{ and}&
  u(\vec{x}) &= g(\vec{x})
    && \text{for }\vec{x}\in\partial\Omega.
\end{align}
In the following examples the exact convex solution is known and the boundary function $g$ is given as the restriction of the exact solution to the boundary  $\partial\Omega$.

In the \emph{first example}~\cite{BFO2010,DG2006,FO2011,FO2011a} the solution is in $C^2(\Omega)$ and radially symmetric. The exact solution and the right-hand side of the Monge-Amp\`ere equation \eqref{eq:example_MA} are given by
\begin{align}\label{eq:example1}
  u(\vec{x}) &= \exp\left(\frac{\norm{\vec{x}}_2^2}{2}\right)
\text{ and}& f(\vec{x}) &= \left(1+\norm{\vec{x}}_2^2\right)\exp\left(\norm{\vec{x}}_2^2\right).
\end{align}

The \emph{second example}, also taken from~\cite{FO2011,FO2011a}, has a solution $u$ in $C^1(\Omega)$ and is defined by
\begin{align}\label{eq:example2}
  u(\vec{x}) &= \frac{1}{2}\left(\max\{0,\norm{\vec{x}-\vec{x}_0}_2-0.2\}\right)^2
\text{ and} & f(\vec{x}) &= \max\left\{0, 1-\frac{0.2}{\norm{\vec{x}-\vec{x}_0}_2}\right\}.
\end{align}

A solution which is in $C^2(\Omega)$ but whose gradient has a singularity near $(1,1)^T\in\partial\Omega$ has also been discussed in~\cite{BFO2010,DG2004,DG2006,FO2011,FO2011a}. This \emph{third example} is given by the exact solution
\begin{align}\label{eq:example3}
  u(\vec{x}) &= -\sqrt{2-\norm{\vec{x}}_2^2}
  & &\text{with right-hand side}
  & f(\vec{x}) &= 2\left(2-\norm{\vec{x}}_2^2\right)^{-2}.
\end{align}
The solution $u$ is also in $W^{1,p}(\Omega)$ for any $p\in[1,4)$; see also~\cite{DG2006}.

In the \emph{fourth example}, also taken from~\cite{BFO2010,FO2011,FO2011a}, the solution is only Lipschitz continuous, i.e., the solution $u$ is in $C^{0,1}(\Omega)$, and is defined by
\begin{align}\label{eq:example4}
  u(\vec{x}) &= \norm{\vec{x}-\vec{x}_0}_2
  & &\text{with right-hand side}
  & f &= \pi \delta_{\vec{x}_0},
\end{align}
where $f$ is defined by the Dirac delta distribution. Note that $u$ is an Aleksandrov solution; see~\cite{FO2011,FO2011a} for the details. In~\cite{BFO2010,FO2011,FO2011a} the distribution is approximated by a piecewise constant function.
On a ball of radius $\nicefrac{h}{2}$, where $h$ is the spatial resolution of the grid, the approximation $f_h$ takes a value such that integral over the ball is conserved. This leads to
\begin{align*}
  f_h(\vec{x}) &:=
    \begin{cases}
      \nicefrac{4}{h^2}, & \text{for }\norm{\vec{x}-\vec{x}_0}_2\leq\nicefrac{h}{2},\\
      0,             & \text{otherwise.}
    \end{cases}
\end{align*}

For the \emph{fifth and last example} the exact solution is unknown such that we cannot compare the results with the exact solution. Nevertheless, Dean and Glowinski~\cite{DG2003,DG2006a,DG2008} and also Feng and Neilan~\cite{FN2009a} discussed this test case with right-hand side
\begin{align}\label{eq:example5}
  f &:= 1
\end{align}
and Dirichlet homogeneous boundary condition. Feng and Neilan~\cite{FN2009a} remark that there exist a unique convex viscosity solution but no classical one.

\begin{remark}\label{remark:typo}
  Froese and Oberman~\cite{FO2011a} claim the solution of \eqref{eq:example4} to be $u(\vec{x})=\sqrt{\norm{\vec{x}-\vec{x}_0}_2}$ which is neither a convex function nor a solution to this right-hand side, but it seems that they used the correct version of \eqref{eq:example4} in their numerical experiments. This is probably a mistake in writing.
\end{remark}

\begin{remark}\label{remark:existence_uniqueness}
  The existence of a solution is guaranteed for the first four test cases due to their definitions. For the uniqueness we refer to the general result~\cite[Corollary 17.2]{GT1983} for classical solutions, which states for our case in \eqref{eq:example_MA} as follows:

  Let $\Omega\subset\R^2$ be a domain. Suppose $u,v\in C^0(\bar{\Omega})\cap C^2(\Omega)$ are strictly convex and we have $\det(\DD^2 u)=\det(\DD^2 v)$ in $\Omega$ and $u=v$ on $\partial\Omega$. We then have $u\equiv v$ in $\Omega$ as well.
\end{remark}


\subsubsection{Results for the five test cases}\label{subsubsec:experiments_standard}

We now apply our solver to the five test cases introduced in Subsection~\ref{subsubsec:test_cases_five},
that all impose Dirichlet boundary conditions.
We still have to fix $\lambda\geq 0$ for the penalty term of our modified determinant ${\det}_{\lambda}^+$ in \eqref{eq:MA_convexity}, which we use to ensure the ellipticity of the Monge-Amp\`ere equation \eqref{eq:example_MA}. Since this additional term vanishes for the exact solution we should use a large value.  Preliminary tests reveal that $\lambda = 10^3$ is a good choice for all subsequent numerical experiments. Our collocation grid is defined as an equidistant grid of $N\times N$ points for different values of $N\in\N$. Note that the grid points coincide with the knots of the B-splines.

For the first four test cases Benamou, Froese, and Oberman~\cite{BFO2010,FO2011,FO2011a} suggest to use as an initial guess the solution of the Poisson boundary value problem
\begin{align}\label{eq:poisson_for_ini}
  \Delta u = \sqrt{2f}
\end{align}
with the same Dirichlet boundary conditions and for the same right-hand side function $f$ as for the corresponding Monge-Amp\`ere equation. In~\cite{FO2011,FO2011a} this Poisson equation is solved in a preprocessing step and after that the result is convexified by the method of Oberman~\cite{Oberman2008a} to ensure a convex initial guess.
Here we also use same initial guess for all five test cases, but, however, it turned out that our method works well even without convexifying the solution of \eqref{eq:poisson_for_ini}. We therefore omit this step.

The spline collocation method has been used to solve \eqref{eq:poisson_for_ini}. In order to validate the numerical result, we compare the maximum absolute error of the numerical solution $u$ at the $N^2$ collocation points with the exact solution $u^*$. The absolute errors are given in Table~\ref{tab:error}, while the corresponding computing times are denoted in Table~\ref{tab:time}. Note that the computing time measurements indicate the overall time for the computation including the computation of the initial guess and the nested iteration scheme. In Figure~\ref{fig:max_error_and_time} the dependency of the maximum error and the computing time on the number of unknowns are shown in a plot with logarithmic scale on both axes. We observe that the complexity of the solution method is proportional to $N^3$.

\begin{table}[ht]
  \centering
  \caption{Maximum error $\norm{u-u^*}_{\infty}$ for the first four test cases of Subsection~\ref{subsubsec:test_cases_five}.}
  \label{tab:error}
  { \footnotesize
    \newcommand{\mc}[1]{\multicolumn{1}{c|}{#1}}
    \newcommand{\mr}[1]{\multirow{2}{*}{#1}}
    \begin{tabular}{|c|c|c|c|c|}
      \hline
        \multicolumn{1}{|c|}{\mr{$N$}}
                & \mc{\mr{$C^2$ example \eqref{eq:example1}}}
                &  \mc{\mr{$C^1$ example \eqref{eq:example2}}}
                &  \mc{example with}
                &  \mc{\mr{$C^{0,1}$ example \eqref{eq:example4}}}\\
          & & & \mc{blow up \eqref{eq:example3}} &\\
      \hline
       31 & $9.60\cdot 10^{-5}$ & $1.18\cdot 10^{-4}$ & $3.76\cdot 10^{-3}$ & $1.25\cdot 10^{-2}$\\
       45 & $4.53\cdot 10^{-5}$ & $7.90\cdot 10^{-5}$ & $3.21\cdot 10^{-3}$ & $1.10\cdot 10^{-2}$\\
       63 & $2.29\cdot 10^{-5}$ & $4.40\cdot 10^{-5}$ & $2.75\cdot 10^{-3}$ & $9.00\cdot 10^{-3}$\\
       89 & $1.14\cdot 10^{-5}$ & $2.86\cdot 10^{-5}$ & $2.34\cdot 10^{-3}$ & $8.34\cdot 10^{-3}$\\
      127 & $5.58\cdot 10^{-6}$ & $2.37\cdot 10^{-5}$ & $1.97\cdot 10^{-3}$ & $8.50\cdot 10^{-3}$\\
      181 & $2.74\cdot 10^{-6}$ & $1.58\cdot 10^{-5}$ & $1.66\cdot 10^{-3}$ & $8.48\cdot 10^{-3}$\\
      255 & $1.37\cdot 10^{-6}$ & $1.01\cdot 10^{-5}$ & $1.41\cdot 10^{-3}$ & $8.73\cdot 10^{-3}$\\
      361 & $6.84\cdot 10^{-7}$ & $7.27\cdot 10^{-6}$ & $1.18\cdot 10^{-3}$ & $8.68\cdot 10^{-3}$\\
      \hline
    \end{tabular}
  }
\end{table}

\begin{table}[ht]
  \centering
  \caption{Computing time in seconds for the five test cases of Subsection~\ref{subsubsec:test_cases_five} (the wall-clock time on the otherwise idle computer has been measured).}
  \label{tab:time}
  { \footnotesize
    \newcommand{\mc}[1]{\multicolumn{1}{c|}{#1}}
    \newcommand{\mr}[1]{\multirow{2}{*}{#1}}
    \begin{tabular}{|c|d{-1}|d{-1}|d{-1}|d{6}|d{-1}|}
      \hline
        \multicolumn{1}{|c|}{\mr{$N$}}
                & \mc{$C^2$}
                &  \mc{$C^1$}
                &  \mc{example with}
                &  \mc{$C^{0,1}$}
                &  \mc{Viscosity}\\
          & \mc{example \eqref{eq:example1}}
          & \mc{example \eqref{eq:example2}}
          & \mc{blow up \eqref{eq:example3}}
          & \mc{example \eqref{eq:example4}}
          & \mc{solution \eqref{eq:example5}}\\
      \hline
       31 &    0.1 &    0.3 &    0.1 &    4.8 &    0.1\\
       45 &    0.1 &    0.6 &    0.2 &   13.3 &    0.2\\
       63 &    0.2 &    1.2 &    0.4 &   18.9 &    0.3\\
       89 &    0.6 &    3.2 &    0.9 &   58.6 &    0.9\\
      127 &    0.9 &    4.1 &    1.7 &   97.6 &    1.7\\
      181 &    2.7 &   12.8 &    4.9 &  345.2 &    5.1\\
      255 &    5.1 &   24.9 &    9.7 &  639.5 &    9.9\\
      361 &   18.4 &   79.8 &   31.6 & 2193.7 &   32.9\\
      \hline
    \end{tabular}
  }
\end{table}

\begin{figure}[ht]
  \centering
  \includegraphics{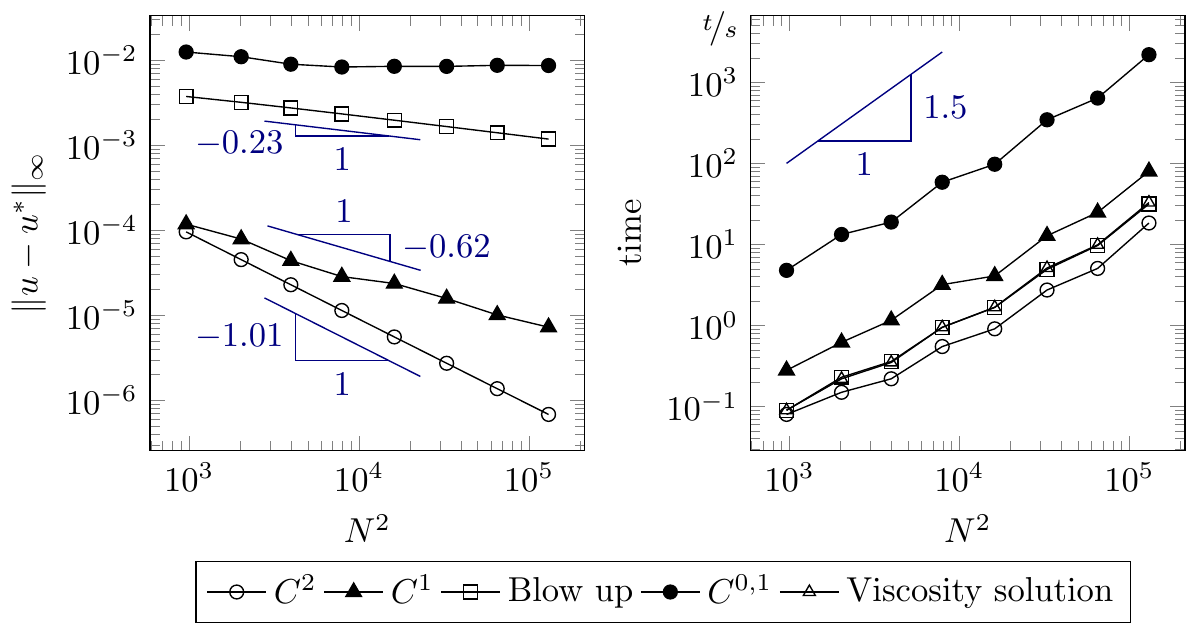}
  \caption{Plot of the maximum errors and computing times for the test cases of Subsection~\ref{subsubsec:test_cases_five} (see also Table~\ref{tab:error} and Table~\ref{tab:time}).}
  \label{fig:max_error_and_time}
\end{figure}

For a better comparison the grid sizes given by $N$ were chosen to match the choices of Froese and Oberman~\cite{FO2011a}. In that paper three of their methods~\cite{BFO2010,FO2011,FO2011a} are compared for the first four test cases, where it turned out that the standard finite difference method~\cite{BFO2010} performed best for the first example.
Using regression analysis for our method we observe that the curve in the double-logarithmic plot has a slope of $-1.01$ which corresponds to a quadratic convergence rate.
This rate agrees with that achieved by the finite difference scheme. Moreover, the differences in the maximum errors is less than a factor of $\nicefrac{3}{2}$.

For the little less smooth solution in the second example the standard finite difference method proposed in~\cite{BFO2010} still leads to smaller errors than the two methods in~\cite{FO2011,FO2011a}. Comparing the absolute errors our method improves the results of the standard finite difference method by a factor of $2$ to $3$ and we observe a slightly superlinear convergence rate.

The third example is a big challenge for the methods because of the blow up of the gradient of the solution at the point $(1,1)^T\in\partial\Omega$. Here the monotone scheme~\cite{FO2011} and the hybrid scheme~\cite{FO2011a} perform best.
Our method shows a convergence that is approximately proportional to the square root of the mesh size. It works more precise than the standard finite difference method by a factor of about $4$ but is not as accurate as the other two schemes whose maximal errors are between $2\cdot 10^{-3}$ and $4\cdot 10^{-5}$. Due to the fact that the schemes~\cite{FO2011,FO2011a} are constructed to converge also to viscosity solutions these methods are suited for less smooth solutions.

The solution of the fourth example does not have a continuous first derivative in $\Omega$ such that it is very difficult to handle even for standard spline interpolation. Here all three methods given in~\cite{BFO2010,FO2011,FO2011a} and also our method do not converge, the error does not drop below $10^{-3}$.
Interestingly, we observe that all three methods, the two methods of~\cite{FO2011,FO2011a}, and our method, stagnate for $N$ larger than $89$. In fact, we do not even expect that our collocation method converges, because it requires the solution to be twice differentiable.

In Figure~\ref{fig:example5} we visualize the result for the fifth test case. In fact, the solution is convex. Figures~\ref{fig:example5_b} and~\ref{fig:example5_c} show cross section of the solution along the $x$-axis and along the diagonal, respectively. These can be compared with those of Dean and Glowinski~\cite{DG2003,DG2006a,DG2008} and Feng and Neilan~\cite{FN2009a}. We observe that both the curvature as well as the minimal values of the functions agree with those in the literature.


\begin{figure}[ht]
  \def\WIDTH{0.32\linewidth}
  \centering
  \subfloat[3d plot.]{
    \includegraphics[width=0.32\linewidth]{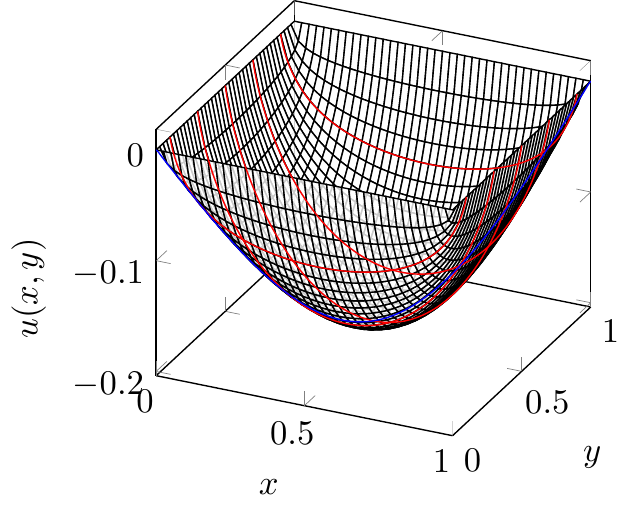}
    \label{fig:example5_a}
  }
  \hfill
  \subfloat[Cross section along $x$-axis.]{
    \includegraphics[width=0.30\linewidth]{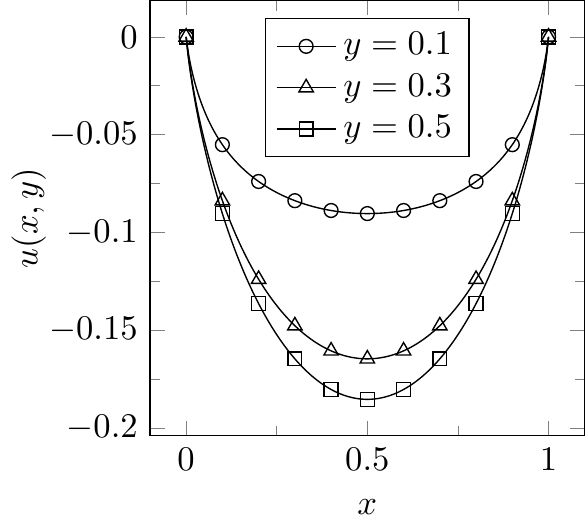}
    \label{fig:example5_b}
  }
  \hfill
  \subfloat[Diagonal cross section.]{
    \includegraphics[width=0.30\linewidth]{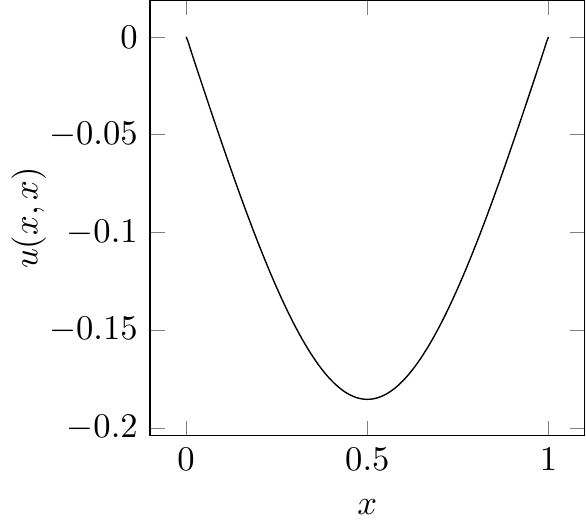}
    \label{fig:example5_c}
  }
  \caption{Solution of the fifth example computed for $N=181$.}
  \label{fig:example5}
\end{figure}


\subsection{Test cases for the inverse reflector problem}\label{subsec:test_cases_inverse_reflector}

First, in Subsection~\ref{subsubsec:reflector_procedure} we give some additional details that are crucial for numerically solving the inverse reflector problem. Then in Subsection~\ref{subsubsec:reflector_setting} we define the geometric setting of our test case for the inverse reflector problem. Afterwards we describe in Subsection~\ref{subsubsec:reflector_initial_guess} how we obtain a good initial guess and in Subsection~\ref{subsubsec:reflector_results} we present the results for some examples.


\subsubsection{Solution procedure}\label{subsubsec:reflector_procedure}

We now briefly discuss the procedure of numerically solving the inverse reflector problem. To this end we focus on three issues that particularly need to be handled to successfully solve Problem~\ref{problem:IR}.

\paragraph{Boundary condition}
In Subsection~\ref{subsec:ma_equation} we saw the mathematical formulation of the reflector problem for the near field and in Subsection~\ref{subsec:reflector_boundary} a relaxation to a sequence of subproblems. For each subproblem we have to solve an equation of Monge-Amp\`ere type with an adjusted boundary function $\phi$. Moreover, we use an iterative nonlinear solver for the solution of each subproblem. To avoid solving the Monge-Amp\`ere equation many times, we intertwine the iterations and update the boundary function $\phi$ immediately after each correction step in the iterative nonlinear solver instead of not updating $\phi$ until the nonlinear solver terminates.

Since we work on a rectangular target set $\Sigma$ we face the problem that the outer normal vector of $\partial\Sigma$ is not defined in a corner. Here we use the normalized sum vector of the two outer normal vectors of both adjacent edges, i.e., the outer normal vectors at the corners of a rectangle $(a,b)\times(c,d)$ are given by the vectors $(\pm \nicefrac{1}{\sqrt{2}},\pm \nicefrac{1}{\sqrt{2}})^T$.

\paragraph{Dark areas on the target}
Let the density function $g$ for our target illumination on $\Sigma$ be given by $8$ bit digital grayscale images, i.e. the gray values of the image are integer values in the range $0, \dots, 255$. Since $g$ is in the denominator on the right-hand side of the Monge-Amp\`ere equation \eqref{eq:MA_reflector}, we require that $g$ is bounded away from zero. This lower bound should be as small as possible. To ensure this constraint
we adjust brightness and contrast of the input image $g$ and consider the image
\begin{align}\label{eq:increas_gray_values}
  \tilde{g}(\vec{Z}):=g(\vec{Z})+\max\{0,20-\min_{\vec{Z'}\in\Sigma}g(\vec{Z'})\}
\end{align}
instead, where the value of $20$ leads to good results. Afterwards this density function needs to be normalized such that the energy conservation \eqref{eq:energy_conservation} holds true.

However, if we try to realize black values in the target image by letting some pixel values go to zero, the right-hand side of the Monge-Amp\`ere equation \eqref{eq:MA_reflector} tends to infinity. The left-hand side is more or less the determinant of the Hessian of $u$; see, e.g., the special case in equation \eqref{eq:MA_reflector_special}, which also needs to go to infinity at some points because of the surjectivity constraint \eqref{eq:MA_reflector_second_boundary}. It follows that the curvature of the reflector must be infinitely large at these points which leads to a kink on its surface.

\paragraph{Nested iteration}
On the one hand our nonlinear solver profits from a good choice of the initial guess. But on the other hand, if we want to produce a very complex image on the target, we need to define our ansatz functions on a very fine grid. Thus we have many degrees of freedom which makes it difficult to obtain a good initial guess.

An efficient way to address this problem is to apply the multilevel technique of Subsection~\ref{subsec:splines_nested}. We therefore start the computation on a very coarse grid of dimension $21\times 21$ and solve the inverse reflector problem. Our target density function $g$ will be given by an image of  size $512\times 512$ pixels. Thus we cannot expect to be able to solve the reflector problem accurately on such a coarse grid. We therefore also coarsen the image. For this reason we define the \emph{standard mollifier function} $\varphi(\vec{x}) := \exp(\nicefrac{-1}{(1-\norm{\vec{x}}_2^2)})$ if $\norm{\vec{x}}_2<1$ and zero otherwise. A discrete approximation of the mollifier function with a support of size $n\times n$ pixels is given by
\begin{align}\label{eq:mollifier}
  \varphi_n(i,j) &:= \frac{\varphi\left(2\frac{i}{n},2\frac{j}{n}\right)}
             {\sum_{r,s\in\Z}\varphi\left(2\frac{r}{n},2\frac{s}{n}\right)}
\end{align}
for $i,j\in\Z$. We now convolve the image with $\varphi_n$ for different $n\in\N$ and solve the problem for these modified images on grids of size $N\times N$ for appropriate $N\in\N$, i.e., we solve the problem many times for different pairs $(N,n)$ to improve the solution; see Figure~\ref{fig:nested_iteration}. We use the following pairs in the given order: $(21,55)$, $(41,55)$, $(41,19)$, $(81,19)$, $(81, 7)$, $(161,7)$, $(161,3)$, $(321,3)$.

\begin{figure}[ht]
  \centering
  \includegraphics[width=0.7\linewidth]{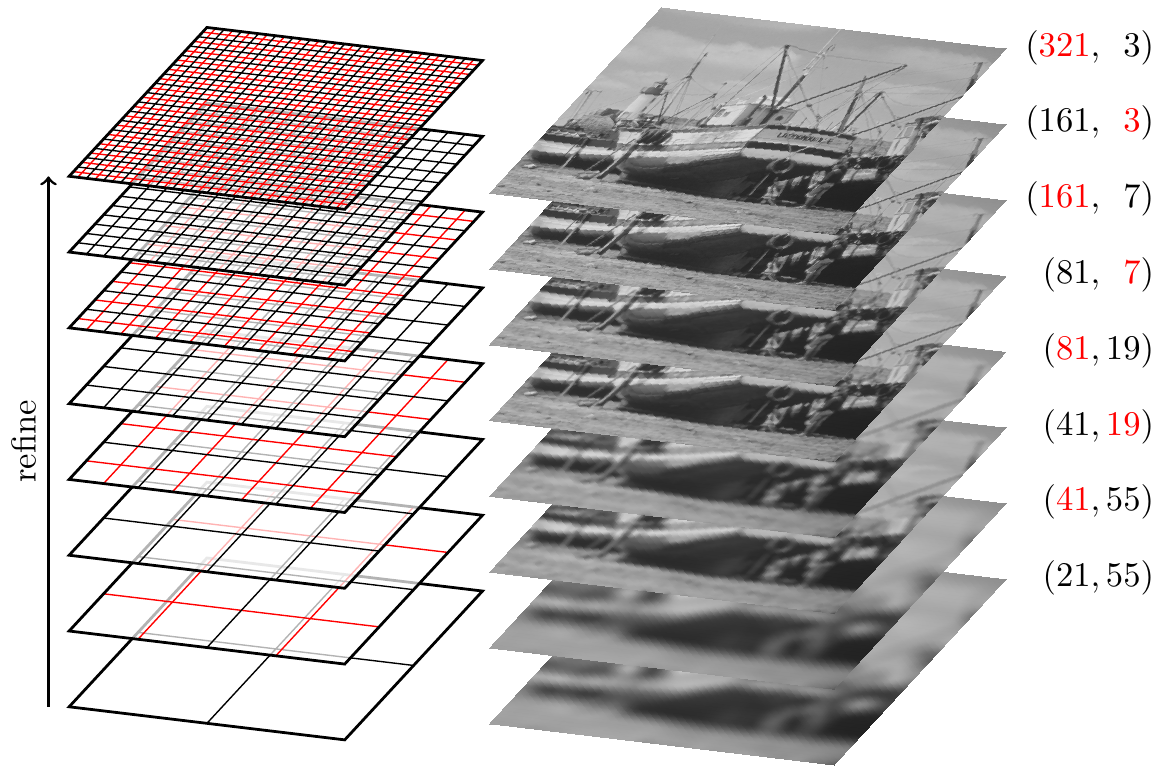}
  \caption{Nested iterations to improve the initial guess from bottom to top with given pairs $(N,n)$, where $N\times N$ is the resolution of the grid and $n$ the parameter for the mollifier $\varphi_n$ in \eqref{eq:mollifier}.}
  \label{fig:nested_iteration}
\end{figure}


\subsubsection{Optical and geometric setting}\label{subsubsec:reflector_setting}

For the illumination of the mirror we choose an isotropic light source, i.e., $f\equiv 1$. Our geometrical setting is given by the target surface, defined by
$ \Sigma:=\{\vec{Z}=(z_1,z_2,z_3)^T\in\R^3\,:\,
    z_1\in\left(-1.5, 1.5\right),\,
    z_2\in (1, 4),\,
    z_3=-5\} $
and by the position of the reflector, which is given in the dimensioned drawing in Figure~\ref{fig:reflector_situation_a}. In the mathematical model, the size of the reflector is controlled by an appropriate constant $\mathcal{G}$ in \eqref{eq:integral_constraint} which in the following examples is  $\mathcal{G}=0.417674$.
For the modified determinant \eqref{eq:modified_det} we again choose the penalty constant $\lambda=10^3$, which leads to good results for all of our examples.


\subsubsection{Initialization}\label{subsubsec:reflector_initial_guess}

The choice of the initial guess is crucial for the convergence of the Newton-type scheme in the collocation method.
Therefore we need an initial guess that is close enough to the solution.
Since there are already other methods available to solve the inverse reflector problem, we can resort to one of them for the initialization.
Due to the nested iteration approach we only need to calculate an initial guess for a strongly blurred input density distribution on a very coarse initial grid.
For that reason we choose the method of supporting ellipsoids~\cite{KO1997,KO1998} for this task, which is a viable choice, because of the low resolution the complexity of the method is not too high; see also Section~\ref{subsec:ellipsoid_method}.

In principle, we need to generate a new initial guess when the desired density function $g$ changes. However, numerical evidence shows that this is not necessary and that we can prepare an universal initial guess that depends on the optical and geometric setup but no longer on $g$.
This is probably possible, because the collocation method starts on a very coarse grid using a strongly blurred and thus an ``almost'' constant version of $g$.

We therefore invoke the method of supporting ellipsoids and compute a reflector surface that produces a constant density function $g$ on the target.
The solution specifies a surface that consists of segments of ellipsoids of revolution. Next we approximate the solution in our B-spline ansatz space corresponding to the coarse initial grid using spline interpolation and we solve the inverse reflector problem again on the same grid with our spline collocation method. The output of the forward simulation of the resulting reflector is shown in Figure~\ref{fig:reflector_situation_b}. As desired the reflector produces a homogeneous illumination pattern on the target. In the following we use this reflector as the initial guess for all calculations in the same geometrical and optical setting but for different target illuminations $g$.


\begin{figure}[ht]
  \def\WIDTH{0.25\linewidth}
  \centering
  \subfloat[Geometric setting]{
    \includegraphics[width=0.22\linewidth]{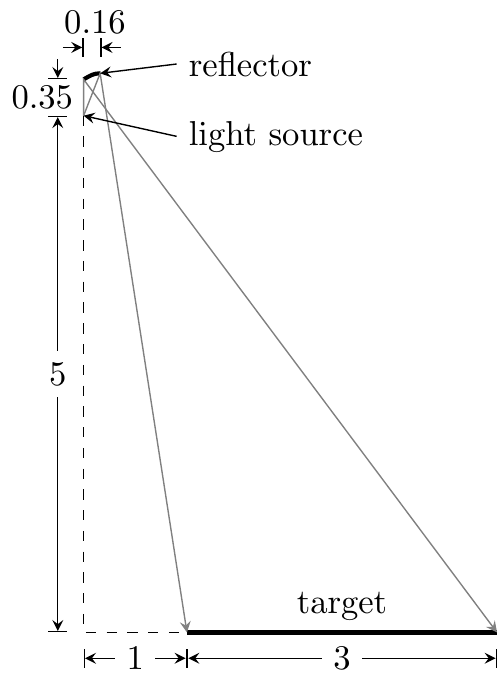}
    \label{fig:reflector_situation_a}
  }
  \hspace{1cm}
  \subfloat[Illumination pattern of the initial guess on the target $\Sigma$.]{
    \includegraphics[width=0.32\linewidth]{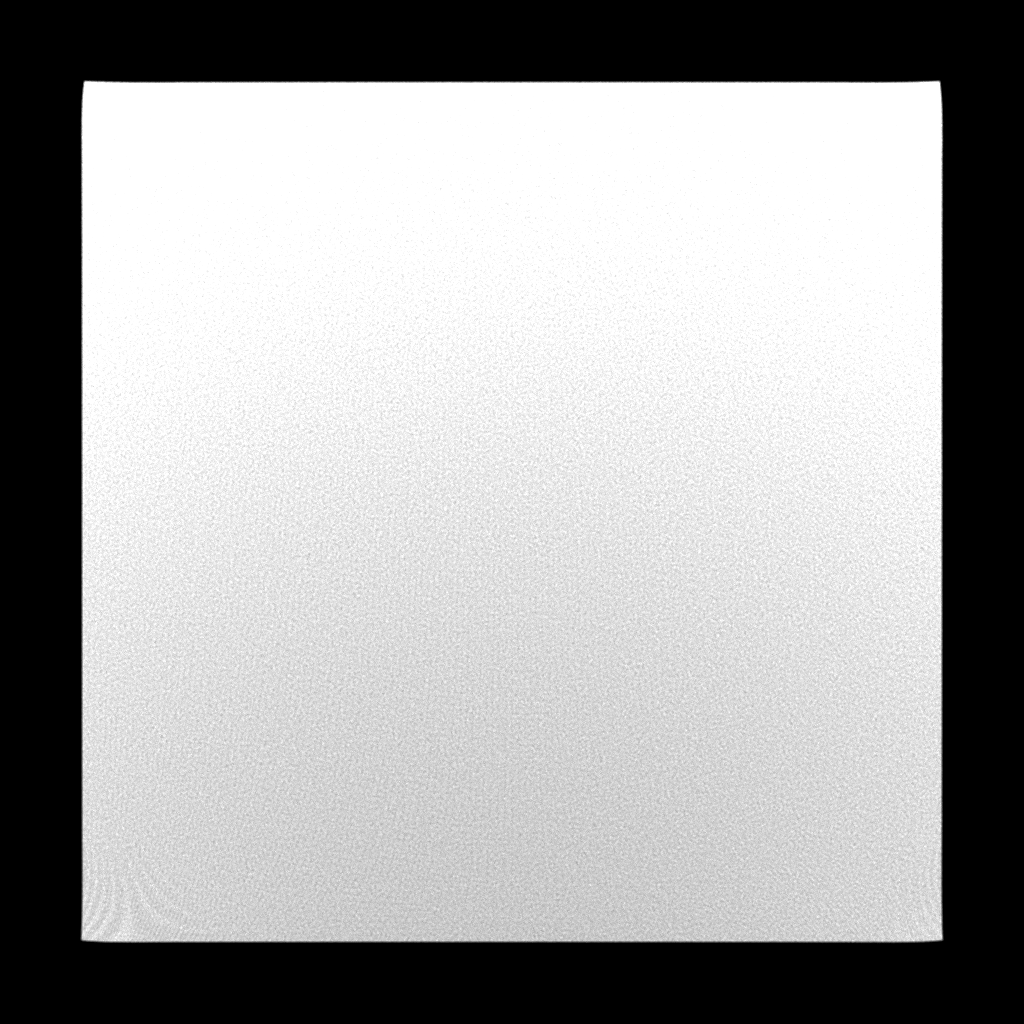}
    \label{fig:reflector_situation_b}
  }
  \caption{Setting of the problem and initial guess.}
  \label{fig:reflector_situation}
\end{figure}


\subsubsection{Results for the inverse reflector problem}\label{subsubsec:reflector_results}

We now calculate the reflector surfaces for three test cases, where the desired target illuminations $g$ are given by three common grayscale test images from the USC-SIPI Image Database~\cite{USC-SIPI}.
In a post-processing step we run the forward simulation by ray tracing to compute the actual illumination pattern produced on the target by the designed reflector surface.
Figure~\ref{fig:reflector_results1} shows the simulation results. Each of the three output images is very close to the corresponding original image and, although the images are slightly blurred, even complex details can easily be identified.

Note that the differences in the illumination between dark areas in the original pictures and the simulations are resulting from the fact that we have to lift dark gray values up; see \eqref{eq:increas_gray_values}. Therefore it is impossible to produce real black areas on the target.

The three test images pose different challenges for our reflector design algorithm.
In the first test image \emph{Boat}, see Figure~\ref{fig:reflector_results1_a}, the problem is to meet the straight lines of the mast, the person standing next to the boat and the lettering on the stern of the vessel. The simulation result looks very good, straight lines are depicted almost perfectly in all directions and the name of the vessel is still readable but only barely.
In our selection the second image \emph{Goldhill}, see Figure~\ref{fig:reflector_results1_b},  represents a different type of pictures. It is rich of different patterns, e.g. the patterns of the roofing tiles and the windows in the foreground as well as the patterns of the trees and bushes in the background, such that we can test how well different patterns are reproduced.
Apart from the slight blurring effect the different patterns are well depicted and can be distinguished easily.
The challenge of the image \emph{Mandrill} is to depict the hair of the beard of the monkey. We can see in Figure~\ref{fig:reflector_results1_c} that our algorithm also passes this test.

\begin{figure}[ht]
  \def\WIDTH{0.32\linewidth}
  \def\HSPACE{-2mm}
  \def\VSPACE{-3mm}
  \centering
  \subfloat[Boat]{
    \parbox{\WIDTH}{
      \includegraphics[width=\linewidth]{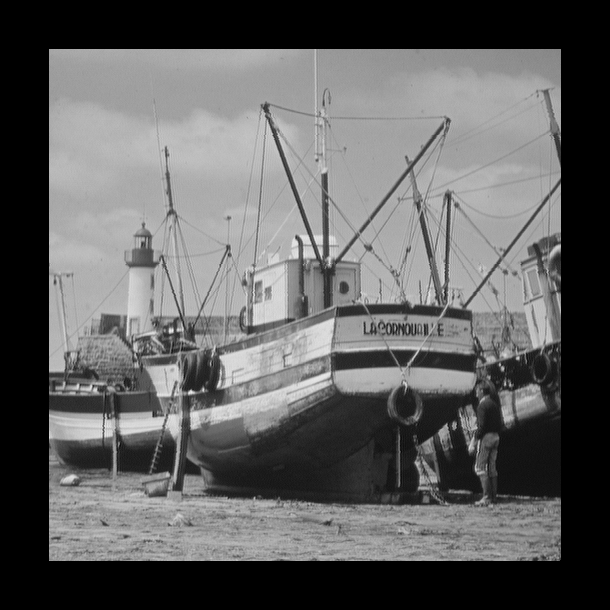}
      \\\vspace{\VSPACE}
      \includegraphics[width=\linewidth]{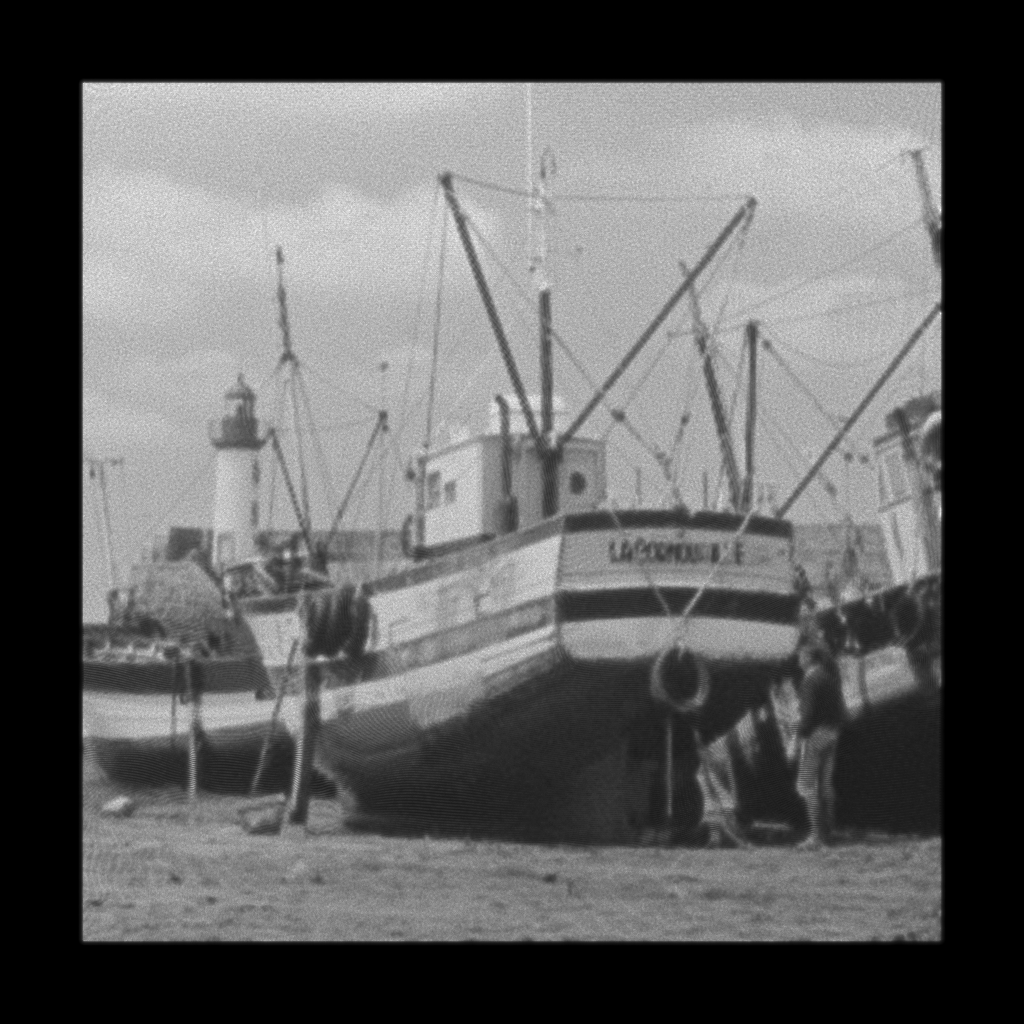}
    }
    \label{fig:reflector_results1_a}
  }
  \hspace{\HSPACE}
  \subfloat[Goldhill]{
    \parbox{\WIDTH}{
      \includegraphics[width=\linewidth]{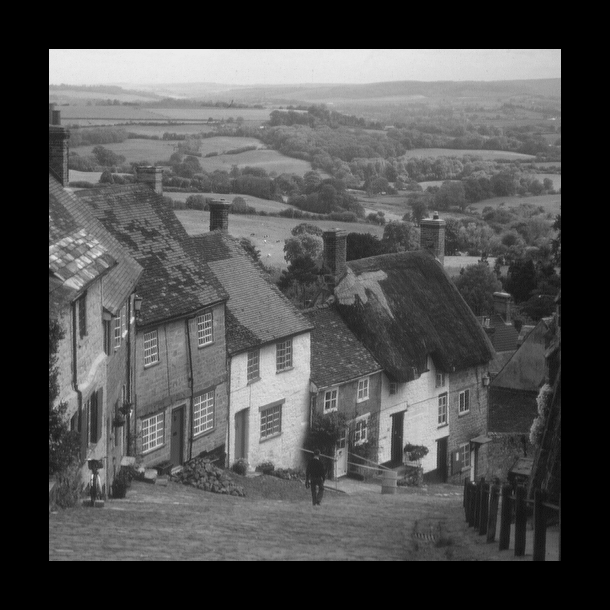}
      \\\vspace{\VSPACE}
      \includegraphics[width=\linewidth]{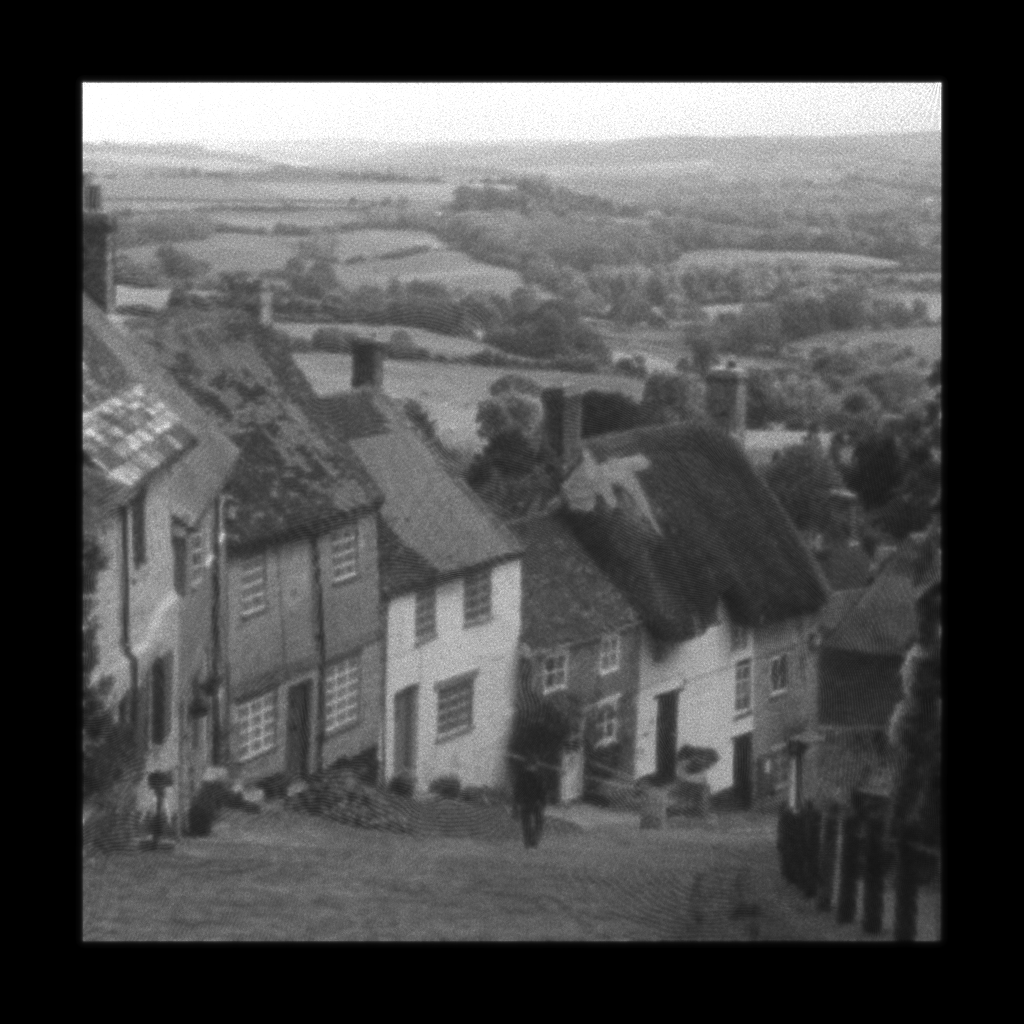}
    }
    \label{fig:reflector_results1_b}
  }
  \hspace{\HSPACE}
  \subfloat[Mandrill]{
    \parbox{\WIDTH}{
      \includegraphics[width=\linewidth]{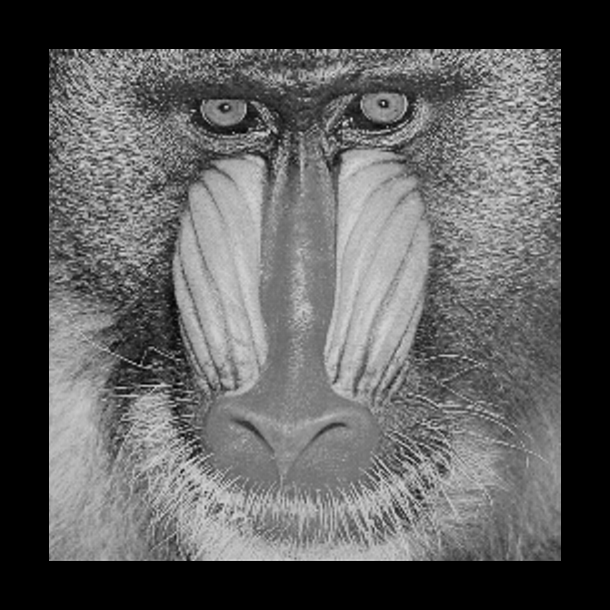}
      \\\vspace{\VSPACE}
      \includegraphics[width=\linewidth]{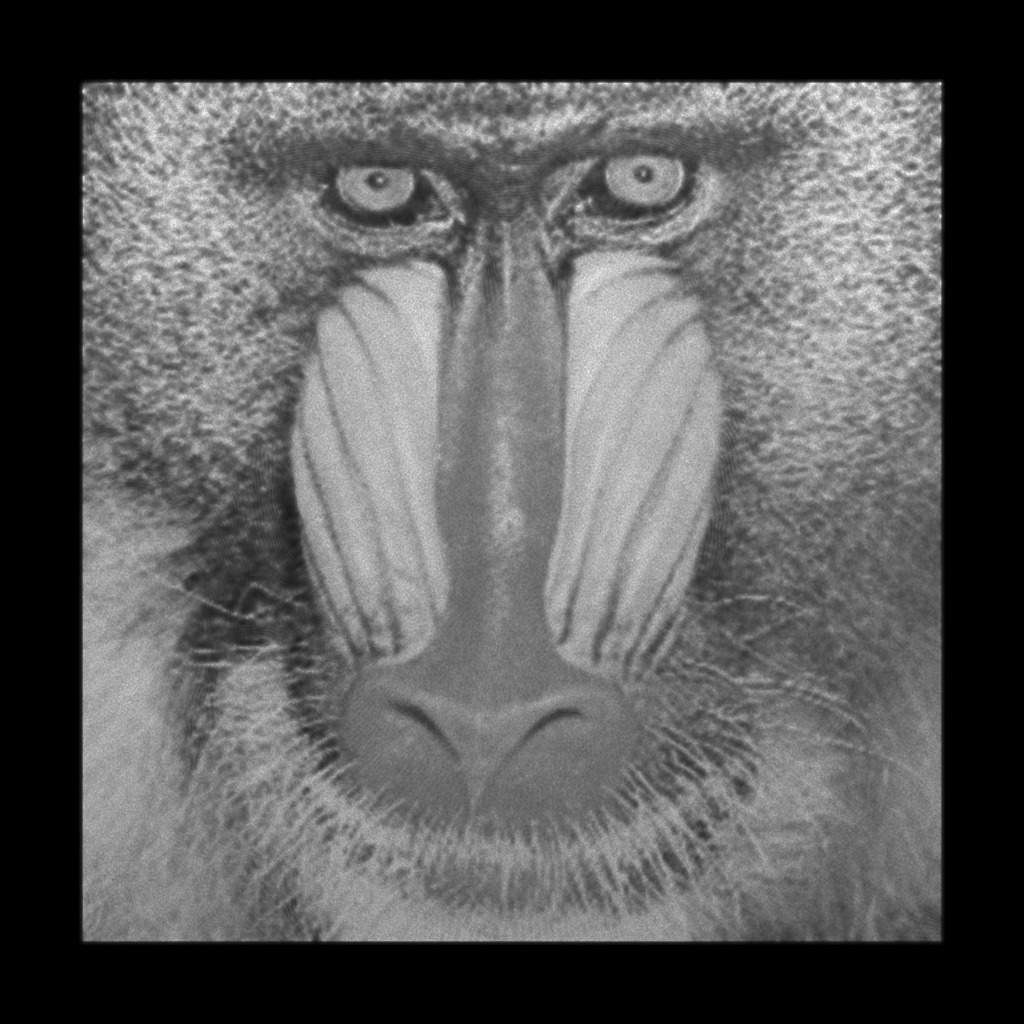}
    }
    \label{fig:reflector_results1_c}
  }
  \caption{Simulation results for three test images. First row: desired distribution (original image);
           second row: distribution after forward simulation by ray tracing).}
  \label{fig:reflector_results1}
\end{figure}

As our final numerical example we compute the surface of a mirror that projects our institute's logo on the screen. Note that this type of cartoon-like images with high contrast and sharp edges is most challenging for our algorithm.
Nevertheless, our algorithm achieves a very good reproduction of the logo; see Figure~\ref{fig:reflector_results4_a} for the desired intensity pattern and Figure~\ref{fig:reflector_results4_b} for the forward simulation result by ray tracing.
In Figure~\ref{fig:reflector_diff_a} we show the position and the coarse shape of the mirror surface, while the fine structure that contains the information of the image is visualized in Figure~\ref{fig:reflector_diff_b} after a high-pass filtering process.
Note that lighter areas on the screen correspond to large areas on the mirror surface.

\begin{figure}[ht]
  \def\WIDTH{0.32\linewidth}
  \centering
  \subfloat[Desired light distribution (original image).]{
    \includegraphics[width=\WIDTH]{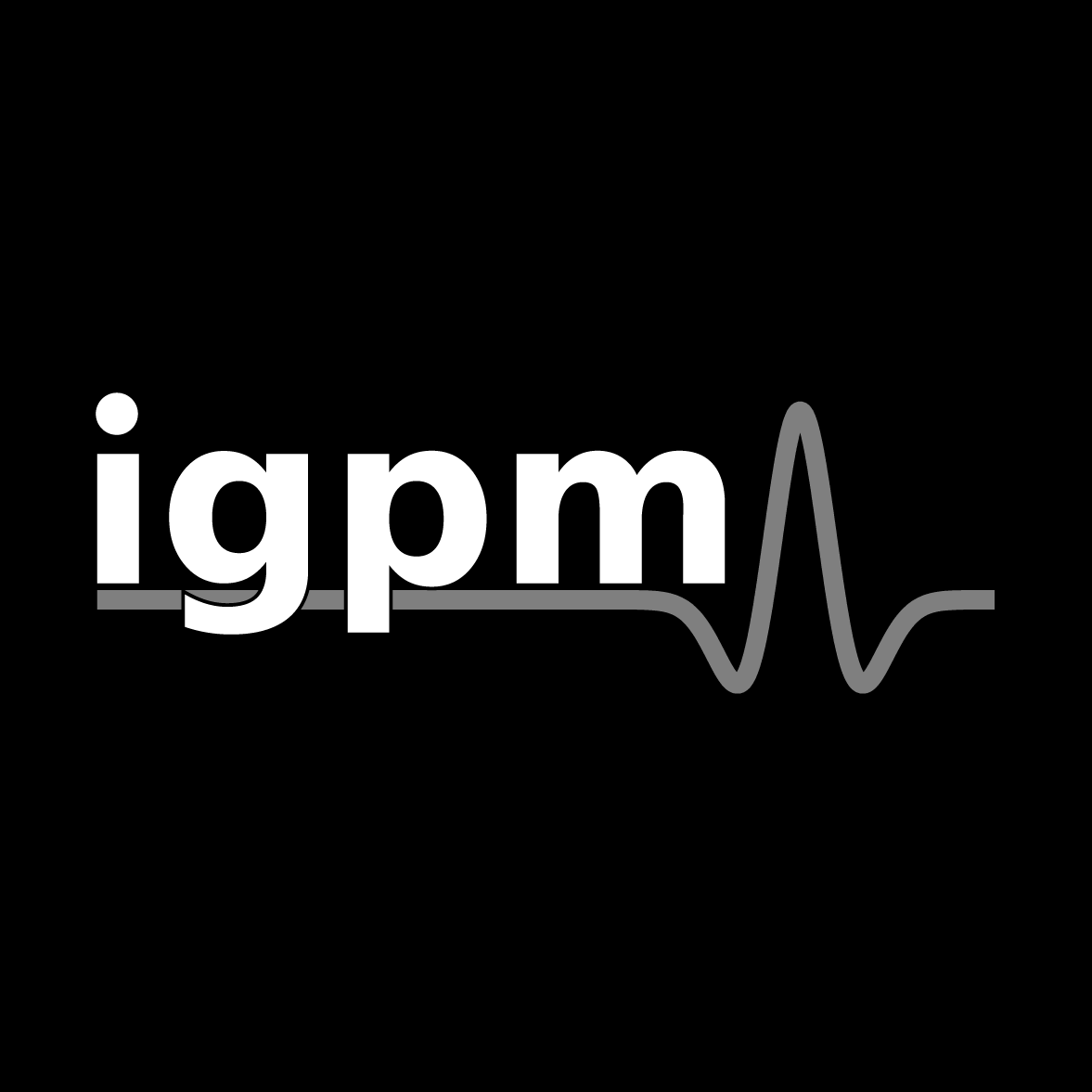}
    \label{fig:reflector_results4_a}
  }
  \hspace{2mm}
  \subfloat[Light distribution after forward simulation by ray tracing (result).]{
    \includegraphics[width=\WIDTH]{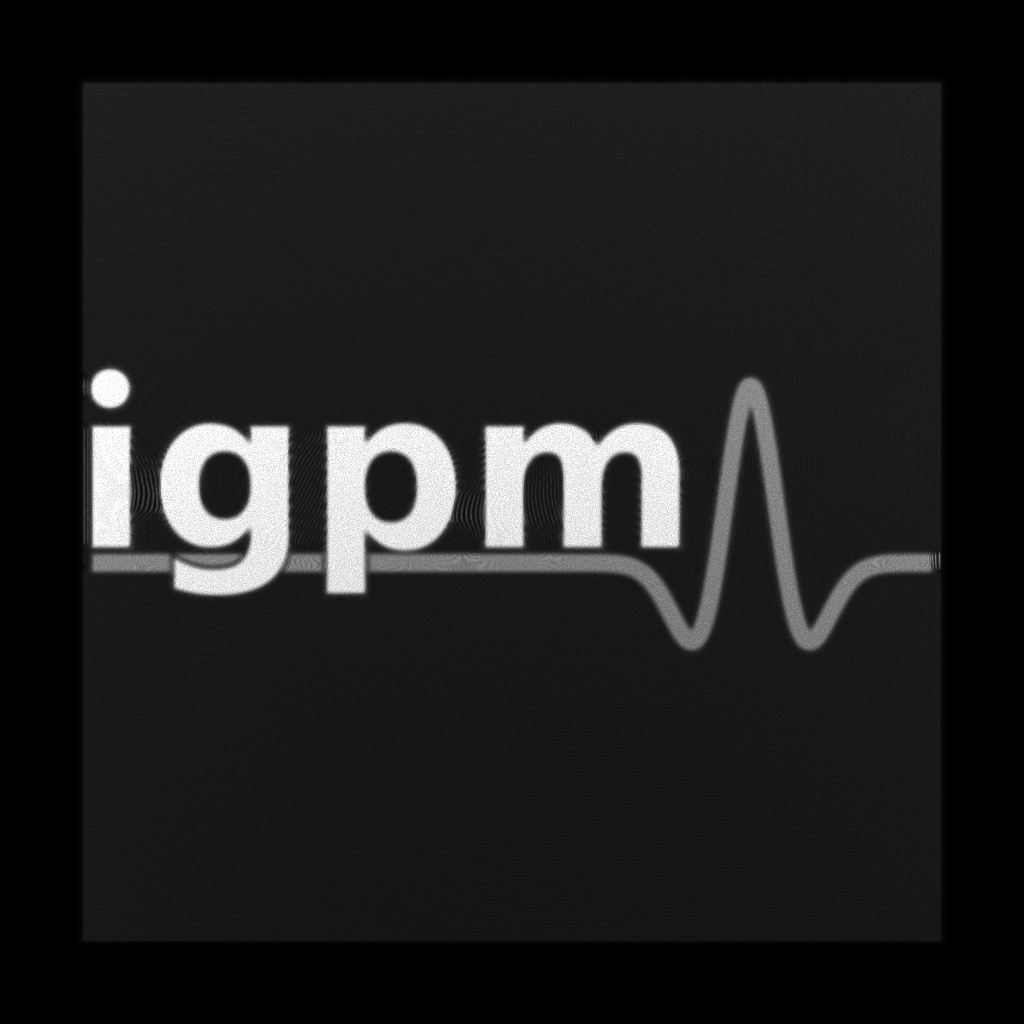}
    \label{fig:reflector_results4_b}
  }\\
  \subfloat[Reflector surface in correct geometrical position (overview)]{
    \includegraphics[width=\WIDTH]{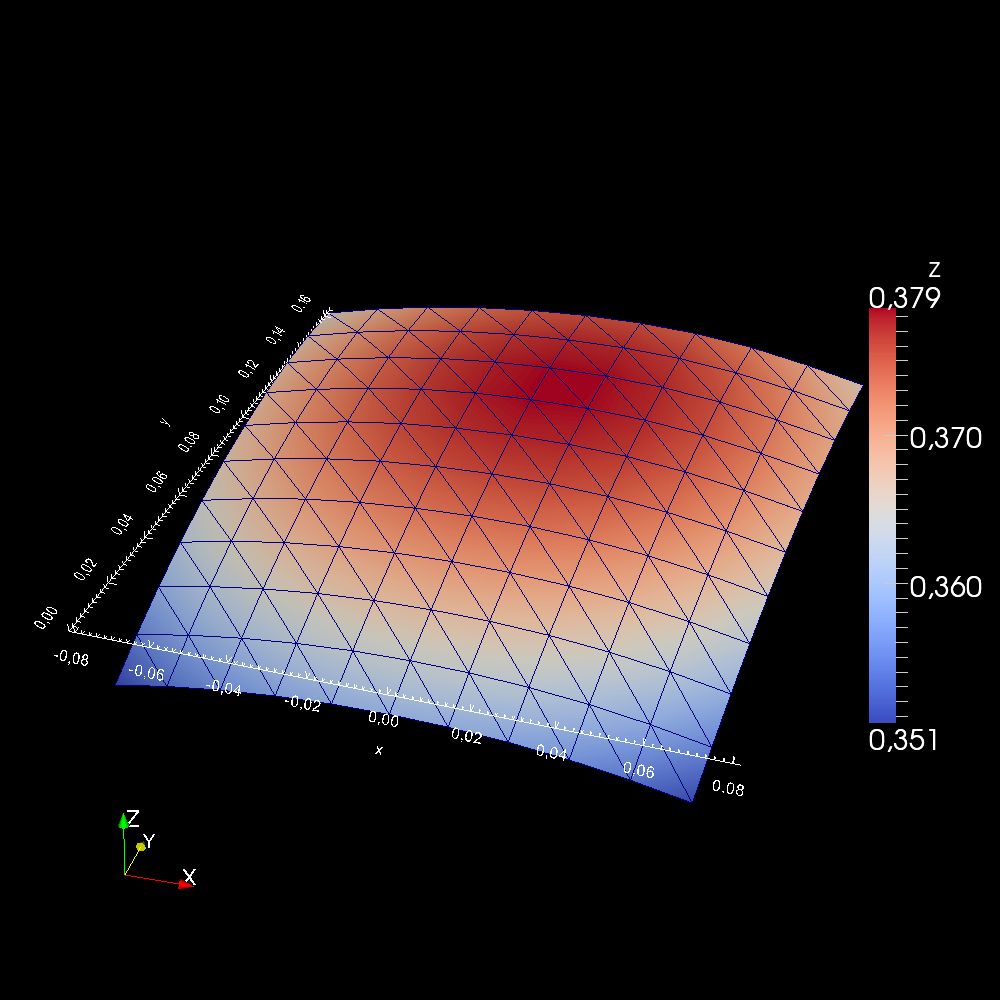}
    \label{fig:reflector_diff_a}
  }
  \hspace{2mm}
  \subfloat[High-frequency components of the reflector (fine structure).]{
    \includegraphics[width=\WIDTH]{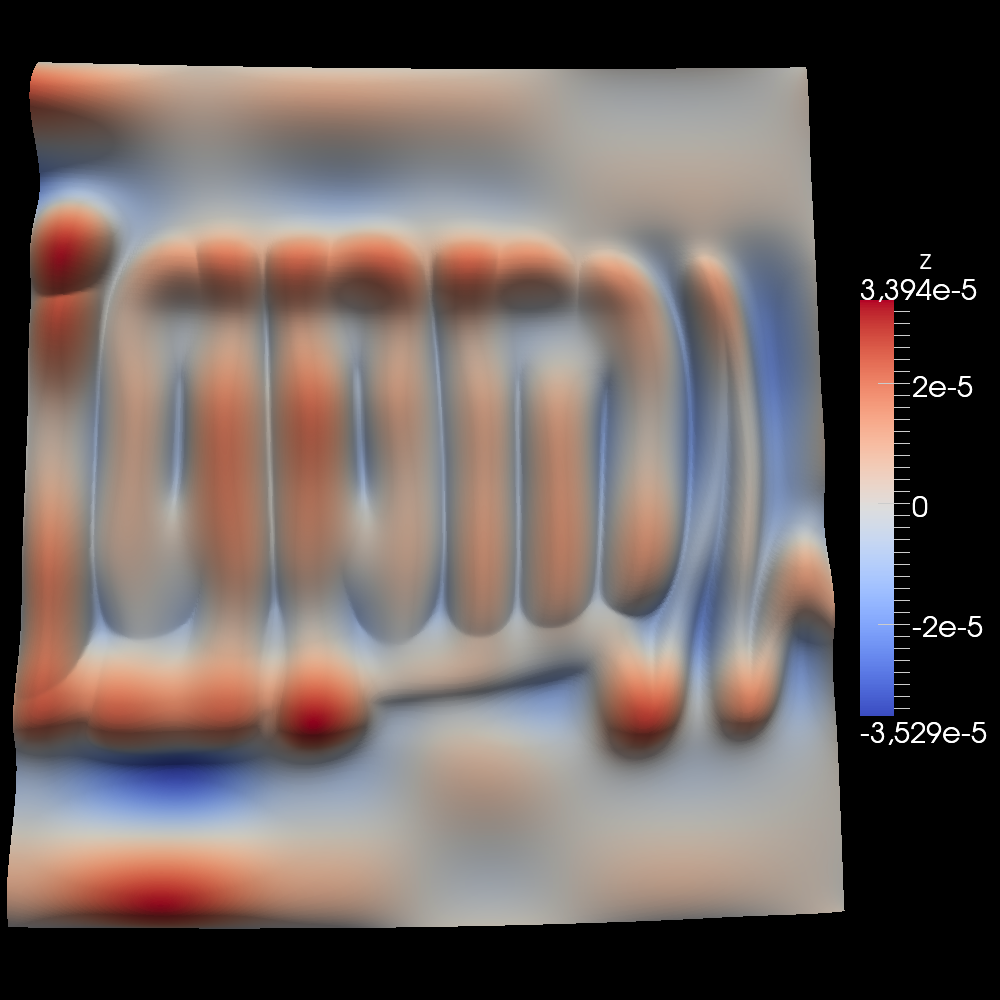}
   \label{fig:reflector_diff_b}
  }
  \caption{Simulation results for our institute's logo.}
  \label{fig:reflector_diff}
\end{figure}


\section{Conclusion and outlook}\label{sec:conclusionoutlook}

We have presented a new B-spline collocation method for the numerical solution of Monge-Amp\`ere type equations, which are strong\-ly nonlinear partial differential equations. Some extensions and manipulations of the equations and boundary conditions have been explained in detail that render it possible to apply the collocation method to the solution of the inverse reflector problem formulated as a Monge-Amp\`ere type equation.

In comparison with existing schemes developed in~\cite{BFO2010,FO2011,FO2011a} our B-spline collocation method produces results that are similar to the finite difference scheme in~\cite{BFO2010}, which is the most accurate method for smooth solutions proposed in these publications.

The largest obstructions encountered for the numerical solution of the inverse reflector problem are how to handle the boundary conditions, how to ensure the uniqueness of the solution, and how to achieve convergence in the numerical scheme. In fact, we explain how these issues can be resolved and that the B-spline collocation method is well-suited even for this strongly nonlinear problem.

In future work the authors plan to extend the numerical method to support other optical devices, in particular lenses. The problem is then, for example, to determine the two surfaces of a lens such that all light passing through this lens is redirected onto the target and produces a prescribed illumination pattern. This is a strongly related problem and can also be modeled by a Monge-Amp\`ere type equation; see, e.g., \cite{Gutierrez2014}.

In modern lighting applications sources can often no longer be assumed to be point sources, e.g., in compact optical systems using light emitting diodes (LEDs). Therefore the challenging question arises how to handle extended light sources in the inverse reflector problem.


\section*{Acknowledgment}

The authors are deeply indebted to Professor Dr. Wolfgang Dahmen for many fruitful and inspiring discussions on the topic of solving equations of Monge-Amp\`ere type.


\end{document}